\documentclass[a4paper]{amsart}
\usepackage[all]{xy}

\usepackage[colorlinks=true]{hyperref}
\usepackage{enumerate}
\usepackage{amssymb}
\usepackage{comment}
\newtheorem{thm}{Theorem}[section]

\newtheorem{prop}[thm]{Proposition}
\newtheorem{lem}[thm]{Lemma}
\newtheorem{cor}[thm]{Corollary}

\theoremstyle{definition}
\newtheorem{defn}[thm]{Definition}
\newtheorem{ex}[thm]{Example}

\theoremstyle{remark}
\newtheorem{rem}[thm]{Remark}

\newcommand{\Rr}{\mathbb R}

\newcommand{\Nn}{\mathbb N}

\renewcommand{\d}{\mathrm d}                           

\renewcommand{\d}{\mathrm d}               

\newcommand{\smc}{\mbox{\,\tiny{$\circ $}\,}}         

\begin{document}
\title{Hierarchies and compatibility on Courant algebroids}

\author{P. Antunes}
\address{CMUC, Department of Mathematics, University of Coimbra, 3001-454 Coimbra, Portugal}
\email{pantunes@mat.uc.pt}
\author{C. Laurent-Gengoux}
\address{UMR 7122, Universit\'e de Metz, 57045 METZ, France
 \newline
 CMUC, Department of Mathematics, University of Coimbra, 3001-454 Coimbra, Portugal}
\email{claurent@univ-metz.fr}
\author{J.M. Nunes da Costa}
\address{CMUC, Department of Mathematics, University of Coimbra, 3001-454 Coimbra, Portugal}
\email{jmcosta@mat.uc.pt}

\begin{abstract}
We introduce several notions (Poisson-Nijenhuis, deforming-Nijenhuis and Nijenhuis pairs) that extend to Courant algebroids the notion of a Poisson-Nijenhuis manifold, by using the idea that both the Poisson and the Nijenhuis structures, although they seem to be different in nature when considered on manifolds, are just $(1,1)$-tensors on the usual Courant algebroid $TM \oplus T^*M $ satisfying several constraints. For each of the generalizations mentioned, we show that there are natural hierarchies obtained by successive deformation by one of the $(1,1)$-tensor.
\end{abstract}

\maketitle

\section*{Introduction}             %
\label{sec:introduction}           %
The purpose of the present article is to explain how $(1,1)$-tensors with vanishing Nijenhuis torsion on a Courant structure naturally give rise to several types of hierarchies - and to show it using as much as possible of supergeometric formalism.  To start with, we say a few words on, respectively, Courant structures, supergeometric formalism, Leibniz algebroids, Nijenhuis torsion and hierarchies. Having recalled these notions, we explain the purpose of our study. We finish this introduction by a more detailed summary of the content of the present work.

\subsection{On Courant structures, supergeometry, Leibniz algebroids, Nijenhuis torsion and hierarchies}

\subsubsection*
{\bf Courant structures.}
It has been noticed by Roytenberg \cite{roythesis} that the original $\Rr$-bilinear skew-symmetric bracket introduced by Courant \cite{cou} on the space of sections of $TM \oplus T^*M $, for $M$ a manifold, can be equivalently defined as:
\begin{equation}  \label{dorfman}
 [(X,\alpha),(Y,\beta)]:=([X,Y], L_X \beta - i_Y \d   \alpha ),
\end{equation}
with $X,Y \in \Gamma(TM)$ and $\alpha,\beta \in \Gamma(T^*M)$. This is a Loday bracket and it was used by Dorfman in \cite{dorfman}. When made abstract, the original Courant's structure on $TM \oplus T^*M $ yields to the definition of Courant algebroid given by Liu, Weinstein and Xu \cite{lwx}, while the version with non-skew-symmetric bracket (\ref{dorfman}) yields to the equivalent definition of Courant algebroid by Roytenberg \cite{roythesis} (see also \cite{YKS05} for a simpler version). Relaxing the Jacobi identity of the Loday bracket, one gets the weaker notion of pre-Courant algebroid (see Definition \ref{def_courant} below).

\vspace{0.5cm}

\noindent {\bf Supergeometric formalism.}
To say the least, to deal with Courant bracket can be an heavy task when it comes about computation (e.g. \cite{YKS92, voronov}), due to the many structures that make it, and to the un-natural aspects of some of its operations. Fortunately, in supergeometric formalism,
all these structures and conditions are encoded in two objects and one condition.
The idea goes as follows. To every vector bundle equipped with a non-degenerate bilinear form is associated a graded commutative algebra, equipped with a Poisson bracket denoted by $\{\cdot, \cdot\} $ (which coincides with the big bracket \cite{YKS92} in some particular cases) \cite{royContemp}. It happens that pre-Courant structures are in one-to-one correspondence with functions of degree $3$. Pre-Courant structures which are indeed Courant are precisely those functions that satisfy:
   $$ \{\Theta,\Theta\}=0$$ (see \cite{royContemp} and \cite{antunes2}).

\vspace{0.5cm}

\noindent {\bf Leibniz algebroids.}
Courant structures on vector bundles can be viewed as special cases of Leibniz algebroids \cite{marrero}. These are vector bundles $E \to M$ equipped with a $\Rr$-bilinear bracket on its space of sections and a vector bundle morphism $\rho:E \to TM$ satisfying the Leibniz rule:
$$[X, fY]=f[X,Y]+ (\rho(X).f)Y$$
and the Jacobi identity:
$$[X,[Y,Z]]=[[X,Y],Z]+[Y,[X,Z]],$$ for all $X,Y,Z \in \Gamma(E)$ and $f \in C^\infty(M)$. Relaxing the Jacobi identity one gets the weaker notion of pre-Leibniz algebroid. When the base manifold reduces to a point, a Leibniz algebroid is just a Leibniz (or Loday) algebra, while a pre-Leibniz algebroid is simply an algebra.
It is easy to check  (see \cite{YKS05}) that pre-Courant algebroids are pre-Leibniz algebroids. But it is important to stress that
the supergeometric approach, referred above for pre-Courant and Courant structures, is not valid in the more general pre-Leibniz and Leibniz algebroid framework.

\vspace{0.5cm}

\noindent {\bf Nijenhuis torsion.}
The Nijenhuis torsion of a $(1,1)$-tensor, i.e. a fiberwise linear endomorphism of $TM$, is the $(2,1)$-tensor given by:
\begin{equation*}\label{eq:modified}
 X,Y \mapsto [NX,NY]   - N[X,Y]_N, \hbox{ where } [X,Y]_N :=  [NX,Y] + [X, NY] -N[X,Y]  .
\end{equation*}
 For a $(1,1)$-tensor, being Nijenhuis torsion-free is in general meaningful (we shall just say ``Nijenhuis tensors" for torsion-free tensors).
  The previous definition can be extended without any change from $TM$ to arbitrary Lie algebroids \cite{magriYKS, graburb},
then from Lie algebroids to Courant algebroids \cite{cgrabm04, YKS11} and Leibniz algebroids \cite{cgrabm04}.

 A general idea about Nijenhuis $(1,1)$-tensors is that it allows to deform an object into an object of the same type,
 for instance, to deform a Lie algebroid bracket $[.,.] $ into the bracket $[.,.]_N $ above, which can be shown to be a Lie algebroid bracket again, or to deform a Poisson structure into another one.

\vspace{0.5cm}

\noindent {\bf Hierarchies.}
There is no mathematical definition of what a \emph{hierarchy} is, but, within the context of integrable systems, the name has been commonly given either to families (indexed by ${\mathbb N}$ or ${\mathbb Z}$) of Hamiltonian functions that commute for a fixed Poisson structure, or of Poisson structures/Lie algebroids which commute between themselves - and sometimes families of both Poisson structures and Hamiltonian functions such that two functions in that family commute with respect to any Poisson structure.
We use that name in the same spirit: i.e., for us a hierarchy is either a family of commuting Courant structures, either a family of Nijenhuis tensors that commute w.r.t. to some Courant structure - or a family of both Courant and (pairs of) Nijenhuis tensors.

A general idea \cite{magriYKS}, \cite{magrimorosi}
about hierarchies is that we start with a few objects, compatible between themselves, then we give ourself a Nijenhuis tensor with the help of which we deform the objects in question, yielding sequences of objects of various types, which are all compatible between themselves.

\subsection{Purpose and content of the present article.}

Our goal is, as we already stated, to construct hierarchies as follows:
\begin{enumerate}
\item hierarchies of Courant structures, given a Nijenhuis tensor on a Courant algebroid,
\item hierarchies of Poisson structures, given a Nijenhuis tensor compatible with a given Poisson structure on a Courant algebroid. For this point, the Courant structure does not need to satisfy the Jacobi identity :
    it just needs to be what we called a pre-Courant structure.
\item hierarchies of Courant structures and pairs of tensors that we call deforming-Nijenhuis pairs or Nijenhuis pairs. Again, pre-Courant structures are enough for most results presented here.
\end{enumerate}
The idea behind item (1) above is simply that what holds true for manifolds and Lie algebroids should hold true for Courant structures as well, and that, in particular, deforming a Courant structure by a Nijenhuis tensor should give a hierarchy of compatible Courant structures. The idea behind items (2) and (3) is more involved. We invite the reader to have in mind the case of Poisson-Nijenhuis structures to get some intuitive picture, but we insist that our constructions apply to much more general contexts. The idea is that, in terms of Courant algebroids, Poisson-Nijenhuis structures  \cite{magrimorosi,magriYKS,graburb} can be seen as follows:
\begin{itemize}
\item we consider the Courant structure $\Theta$ on $TM \oplus T^*M$ already evoked,
\item we see a Poisson structure $\pi$ on the manifold $M$ as a skew-symmetric $(1,1)$-tensor $J_\pi: TM \oplus T^*M \to TM \oplus T^*M $ (see Example~\ref{example1} a),
\item we see a $(1,1)$-tensor $N$ on the manifold $M$ as a skew-symmetric $(1,1)$-tensor $I_N : TM \oplus T^*M \to TM \oplus T^*M$ (see Example~\ref{example1} c),
\end{itemize}
then we check that the conditions of compatibility required on $(\pi,N)$ to be Poisson-Nijenhuis mean that $J_\pi$ and $I_N$ anti-commute and anti-commute w.r.t. the Courant structure, see Example \ref{examplePN}. When made abstract, these conditions yield our Definition \ref{PNpair} of Poisson-Nijenhuis pair on a (pre-)Courant algebroid.
Having established this definition, we can address the purposes of items (2) and (3) above, by generalizing the hierarchies of \cite{magrimorosi}. Indeed, it happens that the notion of Poisson-Nijenhuis is slightly too restrictive, and that hierarchies can be constructed in the more general context of deforming-Nijenhuis pairs and Nijenhuis pairs.

The statements of most results in this article are written in the pre-Courant algebroid framework and are proved using the supergeometric formalism. However, for some of them, the proofs only use the pre-Leibniz structure induced by the pre-Courant structure, so that these results hold not only for pre-Courant algebroids, but also for the more general setting of pre-Leibniz algebroids. This happens, for example, with most results in sections 2.1 and 2.2 and the whole section 4.
Indeed, most results of that section remain true for every vector space endowed with a quadratic form, provided that it admits the property that the deformed operator
by a Nijenhuis torsion-free linear operator is again of the same type - which is true for operators that satisfy 
 relations
like skew-symmetry and Jacobi identity. The lack of convincing examples prevented us from going to such a unnecessary level of generality.

\

Let us give a more precise content of the article.
In section 1, we make a brief introduction of the supergeometric setting for (pre-)Courant structures and we recall the notions of deforming and Nijenhuis tensors.

In section 2, we
show that a Courant structure $\Theta$ can be deformed $k$ times
by a Nijenhuis tensor $I$, and that the henceforth obtained objects $(\Theta_k)_{k \in \Nn}$ are compatible (Theorem \ref{hierarchycourant}).
Then, we show that the property of being compatible is, for a given compatible pair $(I,J)$, also preserved when deforming $n$ times $J$ by $I$, provided that $I$ is Nijenhuis (or at least satisfies a weaker condition involving the vanishing of torsion of $I$ on the image of $J$), and that this result still holds true with respect to pre-Courant structures $\Theta_k$ obtained when deforming $\Theta$ by $I$ (Theorem~\ref{new thma}). An even more general case is obtained when considering the tensor $I^{2s+1}$, $s \in \Nn$, which is the deformation of $I$ by itself an odd number of times, and, if $J$ is also Nijenhuis, $J$ is replaced by  $I^n \smc J^{2m+1}$, $n,m \in \Nn$ (Theorem \ref{propconcgeral}).

In section 3, we turn our attention to deforming-Nijenhuis pairs, i.e. compatible pairs $(J,I)$ where $J$ is a deforming tensor and $I$ is Nijenhuis for $\Theta$. We show that if $(J,I)$ is a deforming-Nijenhuis pair for $\Theta$, then  $(J,I^{2n+1})$ is a deforming-Nijenhuis pair for $\Theta_{k}$, for all $k,n\in \Nn$ (Theorem \ref{oddpowerI}).
Then, we consider
Poisson-Nijenhuis pairs $(J,I)$, i.e. deforming-Nijenhuis pairs where the deforming tensor $J$ is supposed to be Poisson for $\Theta$, and we state one of the main results of the article, which is the construction of a hierarchy of Poisson-Nijenhuis pairs for $\Theta_{k}$, for all $k \in \Nn$, that includes pairs of compatible Poisson tensors (Theorem~\ref{generalPNhierarchy}).

Last, in section 4, we conclude with the case of Nijenhuis pairs, i.e. pairs $(I,J)$ of Nijenhuis tensors compatible w.r.t. to $\Theta$.
More precisely, we show that if
$(I,J)$ is a Nijenhuis pair for $\Theta$, then for all $m,n,t \in \Nn$, $(I^{2m+1}\smc J^n, J^{2t+1})$ is a Nijenhuis pair for $\Theta$, and, more generally, for all the Courant structures obtained by deforming $\Theta $ several times, either by $I$ or by $J$ (Theorem \ref{gen_hierarchy}).


\section{Skew-symmetric tensors on Courant algebroids}

\subsection{Courant algebroids in supergeometric terms}  \label{subsection:1.1}
We begin this section by introducing the supergeometric setting, following the same approach as in \cite{voronov,roy,royContemp}. Given a vector bundle $A \to M$, we denote by $A[n]$ the graded manifold obtained by shifting the fibre degree by $n$. The graded manifold $T^*[2]A[1]$ is equipped with a canonical symplectic structure which induces a Poisson bracket on its algebra of functions $\mathcal{F}:=C^\infty(T^*[2]A[1])$. This Poisson bracket is sometimes called the \emph{big bracket} (see \cite{YKS92}, \cite{YKS05}).

Let us describe locally this Poisson algebra. Fix local coordinates $x_i, p^i,\xi_a, \theta^a$, $i \in \{1,\dots,n\}, a \in \{1,\dots,d\}$, in $T^*[2]A[1]$, where $x_i,\xi_a$ are local coordinates on $A[1]$ and $p^i, \theta^a$ are their associated moment coordinates. In these local coordinates, the Poisson bracket is given by
 $$ \{p^i,x_i\}=\{\theta^a,\xi_a\}=1,  \quad  i =1, \dots, n, \, \, a=1, \dots , d, $$
while all the remaining brackets vanish.

The Poisson algebra of functions $\mathcal{F}$ is endowed with a $(\mathbb{N} \times \mathbb{N})$-valued bidegree. We define this bidegree locally but it is well defined globally (see \cite{voronov, roy} for more details). The bidegrees are locally set as follows: the coordinates on the base manifold $M$, $x_i$, $i \in \{1,\dots,n\}$, have bidegree $(0,0)$, while the coordinates on the fibres, $\xi_a$, $a \in \{1,\dots,d\}$, have bidegree $(0,1)$ and their associated moment coordinates, $p^i$ and $\theta^a$, have bidegrees $(1,1)$ and $(1,0)$, respectively.
The algebra of functions $\mathcal{F}$ inherits this bidegree and
$$   \mathcal{F}=\bigoplus_{k,l \in \mathbb{N} \times \mathbb{N}} \mathcal{F}^{k,l}, $$
where $\mathcal{F}^{k,l}$ is the space of functions of bidegree $(k,l)$.
We can verify that the big bracket has bidegree $(-1,-1)$, i.e.,
$$\{\mathcal{F}^{k_1,l_1},\mathcal{F}^{k_2,l_2}\}\subset \mathcal{F}^{k_1+k_2-1,l_1+l_2-1}.$$

This construction is a particular case of a more general one in which we consider a vector bundle $E$ equipped with a fibrewise non-degenerate symmetric bilinear form $\langle.,.\rangle$.
In this more general setting, we consider the graded symplectic manifold $\mathcal{E}:=p^*(T^*[2]E[1])$, which is the pull-back of $T^*[2]E[1]$ by the application $p:E[1] \to E[1]\oplus E^*[1]$ defined by $X \mapsto (X, \frac{1}{2}\langle X,.\rangle)$. We denote by $\mathcal{F}_{E}$ the graded algebra of functions on $\mathcal{E}$, i.e., $\mathcal{F}_{E}:=C^\infty(\mathcal{E})$. The algebra of functions $\mathcal{F}_{E}$ is equipped with the canonical Poisson bracket, denoted by $\{.,.\}$, which has degree $-2$. Notice that $\mathcal{F}_{E}^0=C^\infty(M)$ and $\mathcal{F}_{E}^1=\Gamma(E)$. Under these identifications, the Poisson bracket of functions of degrees $0$ and $1$ is given by
$$\{f,g\}=0,\; \; \{f, X\}=0 \quad {\hbox{and}} \quad \{X,Y\}=\langle X,Y \rangle,$$
for all $X,Y \in \Gamma(E)$ and $f,g \in C^\infty(M)$.

 The construction described in the beginning of this section corresponds to the case where $E:=A\oplus A^*$ and $\langle .,.\rangle$ is the usual symmetric bilinear form. Notice that, with the notation introduced so far, the algebra of functions $\mathcal{F}=C^\infty(T^*[2]A[1])$ should be denoted by $\mathcal{F}_{A\oplus A^*}$.

\

Let us define the notion of (pre-)Courant structure on a vector bundle $E$ equipped with a fibrewise non-degenerate symmetric bilinear form $\langle.,.\rangle$.

\begin{defn} \label{def_courant}
 A \emph{pre-Courant} structure on $(E, \langle.,.\rangle)$ is a pair $(\rho, [.,.])$, where the \emph{anchor} $\rho$ is a bundle map from $E$ to $TM$ and the \emph{Dorfman bracket} $[.,.]$ is a $\mathbb{R}$-bilinear (non necessarily skew-symmetric) assignment on $\Gamma(E)$ satisfying the relations
\begin{equation} \label{pre_Courant1}
\rho(X)\cdot\langle Y,Z\rangle=\langle[X,Y],Z\rangle +  \langle Y,[X,Z]\rangle
\end{equation}
and
\begin{equation} \label{pre_Courant2}
\rho(X)\cdot \langle Y,Z\rangle=\langle X, [Y,Z]+ [Z,Y]\rangle,
\end{equation}
for all $X,Y,Z \in \Gamma(E)$. \footnote{ From (\ref{pre_Courant1}) and (\ref{pre_Courant2}), we get \cite{YKS05}
$$[X, fY]=f[X,Y]+ (\rho(X).f)Y,$$ for all $X,Y \in \Gamma(E)$ and $f \in C^\infty(M)$. Thus, as we already mentioned in the Introduction, a pre-Courant algebroid is always a pre-Leibniz algebroid.}

Moreover, if the Jacobi identity,
$$[X,[Y,Z]] =[[X,Y],Z] + [Y,[X,Z]],$$
is satisfied for all $X,Y,Z \in \Gamma(E)$, then the \emph{Dorfman bracket} $[.,.]$ is a Leibniz bracket and the pair $(\rho, [.,.])$ is called a \emph{Courant} structure on \mbox{$(E,\langle.,.\rangle)$}.
\end{defn}

\

There is a one-to-one correspondence between pre-Courant structures on $(E, \langle.,.\rangle)$ and functions of $\mathcal{F}_E^3$. The anchor and Dorfman bracket associated to a given $\Theta\in \mathcal{F}_E^3$ are defined, for all $X,Y \in \Gamma(E)$ and $f \in C^\infty(M)$, by
$$\rho(X)\cdot f=\{\{X,\Theta\},f\} \quad {\hbox{and}} \quad {[X,Y]=\{\{X,\Theta\},Y\}}.$$

The following theorem addresses how the Jacobi identity is expressed in this supergeometric setting.

\begin{thm} \cite{royContemp}
There is a one-to-one correspondence between Courant structures on \mbox{$(E,\langle.,.\rangle)$} and functions $\Theta \in \mathcal{F}_E^3$ such that \mbox{$\{\Theta,\Theta\}=0$}.
\end{thm}

If $\Theta$ is a (pre-)Courant structure on $(E, \langle.,.\rangle)$, then the triple $(E, \langle.,. \rangle, \Theta)$ is called a \emph{(pre-)Courant algebroid}.
For the sake of simplicity, we will often denote a (pre-)Courant algebroid by the pair $(E, \Theta)$ instead of the triple $(E, \langle.,. \rangle, \Theta)$.

When $E= A \oplus A^*$ and $\langle.,.\rangle$ is the usual symmetric bilinear form, a pre-Courant structure $\Theta \in \mathcal{F}_E^3$ can be decomposed as a sum of homogeneous terms with respect to its bidegrees:
$$\Theta=\mu + \gamma + \phi + \psi,$$
with $\mu \in \mathcal{F}_{A \oplus A^*}^{1,2}, \gamma \in \mathcal{F}_{A \oplus A^*}^{2,1}, \phi \in \mathcal{F}_{A \oplus A^*}^{0,3}=\Gamma(\bigwedge^3 A^*)$ and $\psi \in \mathcal{F}_{A \oplus A^*}^{3,0}=\Gamma(\bigwedge^3 A)$.

We recall from~\cite{roythesis} that, when $\gamma = \phi = \psi =0$, $\Theta$ is a Courant structure on $(A \oplus A^*, \langle.,.\rangle)$ if and only if $(A,\mu)$ is a Lie algebroid.
Also, when $\phi = \psi =0$, $\Theta$ is a Courant structure on $(A \oplus A^*, \langle.,.\rangle)$ if and only if $\left((A,\mu),(A^*,\gamma)\right)$ is a Lie bialgebroid.

\

\subsection{Deformation of Courant structures by skew-symmetric tensors}

Let $(E,\langle .,. \rangle,\Theta)$ be a pre-Courant algebroid and $J:E \to E$  a vector bundle endomorphism of $E$.
The \emph{deformation} of the Dorfman bracket $[.,.]$ by $J$ is defined, for all sections $X,Y$ of $E$, by
$$[X,Y]_{J} =[JX,Y]+[X,JY]-J[X,Y].$$

The $(1,1)$-tensors on $E$ will be seen as vector bundle endomorphisms of $E$. A $(1,1)$-tensor $J:E \to E$ is said to be {\em skew-symmetric} if $$\langle Ju,v \rangle + \langle u,Jv \rangle =0,$$
for all $u,v \in E$.
If we consider the endomorphism $J^*$ defined by $ \langle u, J^*v \rangle = \langle Ju,v \rangle $, then $J$ is skew-symmetric if and only if  $J+J^*=0$.
If $J$ is skew-symmetric, then $[.,.]_J$ satisfies (\ref{pre_Courant1}) and (\ref{pre_Courant2}), so that $(\rho \smc J, [.,.]_{J})$ is a pre-Courant structure on $(E, \langle.,.\rangle)$ \footnote{In fact, it suffices that $J$ satisfies the condition $J+ J^*= \lambda id_E$, for some $\lambda \in \Rr$, to guarantee that  $(\rho \smc J, [.,.]_J)$ is a pre-Courant structure on $(E, \langle.,.\rangle)$, see \cite{cgrabm04}.}.

When the $(1,1)$-tensor $J:E \to E$ is skew-symmetric, the deformed pre-Courant structure $(\rho \smc J, [.,.]_J)$ is given, in supergeometric terms, by $\Theta_{J}:=\{J,\Theta\}\in\mathcal{F}_E^{3}$.
The deformation of $\Theta_J$ by the skew-symmetric $(1,1)$-tensor $I$ is denoted by $\Theta_{J,I}$, i.e. $\Theta_{J,I}=\{I,\{J, \Theta \}\}$, while the deformed Dorfman bracket $([.,.]_{J})_{I}$ is denoted by $[.,.]_{J,I}$.

Recall that a vector bundle endomorphism $I:E \to E$ is a {\em Nijenhuis} tensor on the Courant algebroid $(E, \Theta)$ if its torsion vanishes. The torsion ${\mathcal T}_{\Theta}I$ is given, for all $X,Y \in \Gamma (E)$, by

\begin{equation*} \label{torsion}
{\mathcal T}_{\Theta}I (X,Y)= [IX,IY]-I([X,Y]_{I})
\end{equation*}
or, equivalently, by
\begin{equation} \label{second_def_torsion}
{\mathcal T}_{\Theta}I (X,Y)=\frac{1}{2}([X,Y]_{{I,I}}-[X,Y]_{{I^2}}),
\end{equation}
where $I^2=I \smc I$.
When $I$ is skew-symmetric and $I^2= \alpha\, id_E$, for some $\alpha \in \Rr$, (\ref{second_def_torsion}) can be written, in supergeometric terms, as follows \cite{grab}:
\begin{equation} \label{supergeometric_torsion}
{\mathcal T}_{\Theta}I= \frac{1}{2}(\Theta_{I,I}-\alpha \Theta).
\end{equation}

When $(E, \Theta)$ is a pre-Courant algebroid, the definition of Nijenhuis tensor is the same as in the case of a Courant algebroid.

\begin{ex}\label{ex:central}
For every Lie algebra ${\mathcal G}$, any linear operator $I$ valued in the center and such that the kernel of $I^2$ contains the commutator
$ [{\mathcal G}, {\mathcal G}]$ is a Nijenhuis operator.
\end{ex}

The notion of {\em deforming} tensor for a Courant structure $\Theta$ on $E$ was introduced in \cite{YKS11}. The definition holds in the case of a pre-Courant algebroid and it will play an important role in this article.
\begin{defn}  \label{deforming}
A skew-symmetric $(1,1)$-tensor $J$ on $(E,\Theta)$ is said to be {\em deforming for} $\Theta$ if $\Theta_{J,J} =\eta \Theta$, for some $\eta \in \Rr$.
\end{defn}

\begin{rem}
If $I$ is Nijenhuis for $\Theta$ and satisfies $I^2= \eta id_E$, for some $\eta \in \Rr$, then, from (\ref{supergeometric_torsion}), we have ${\mathcal T}_{\Theta}I=0 \Rightarrow \Theta_{I,I} =\eta \, \Theta$, i.e. $I$ is deforming for $\Theta$. This was also noticed in \cite{YKS11}.
\end{rem}

When $E= A \oplus A^*$ and $\langle.,.\rangle$ is the usual symmetric bilinear form, a skew-symmetric $(1,1)$-tensor $J:A \oplus A^* \to A \oplus A^*$ is of the type
\begin{equation}  \label{skewJ}
J= \left(
\begin{array}{cc}
 N & \pi^{\sharp} \\
\omega^{\flat} & -N^*
\end{array}
\right),
\end{equation}
with $N: A \to A, \pi \in \Gamma(\bigwedge^2 A)$ and $\omega \in \Gamma(\bigwedge^2 A^*)$. In the supergeometric framework, $J$ corresponds to the function $N+\pi+\omega$, which we also denote by $J$. Therefore, we have $\Theta_{J}=\{N+\pi+\omega, \Theta\}$.

Now, we present several examples of skew-symmetric tensors which are deforming or/and Nijenhuis, in the case where $(E=A \oplus A^*, \Theta)$ is a Courant algebroid with $\Theta=\mu$ and $\mu$ a Lie algebroid on $A$.

\begin{ex} \label{example1}

\

{\bf a)} Let $\pi$ be a bivector on $A$ and $J_\pi=
\left(
\begin{array}{cc}
 0 & \pi \\
0 & 0
\end{array}
\right).$  Then, $J_\pi$ is deforming for $\Theta=\mu$ if and only if $\pi$ is a Poisson bivector on the Lie algebroid $(A, \mu)$.

If $\pi$ is a Poisson bivector on $(A, \mu)$ then, denoting by $[.,.]_\mu$ the Gerstenhaber bracket on $\Gamma(\bigwedge^{\bullet} A)$, we have $0=[\pi, \pi]_\mu = \{\pi, \{\pi, \mu\}\}=\mu_{J_\pi, J_\pi}$, so that $J_\pi$ is deforming for $\mu$. If $J_\pi$ is deforming for $\mu$, then $\mu_{J_\pi, J_\pi}= \eta \, \mu$, with $\eta \in \Rr$. Since $\mu$ and  $\mu_{J_\pi, J_\pi}$ do not have the same bidegree, we get $\mu_{J_\pi, J_\pi}= \eta \, \mu \Leftrightarrow (\eta =0$ and  $\{\pi, \{\pi, \mu\}\}=0)$. Thus, $\pi$ is a Poisson bivector on the Lie algebroid $(A, \mu)$.

\

{\bf b)} Let $J_\pi$ be as in a).   Then, $J_\pi$ is Nijenhuis for $\Theta=\mu$ if and only if $\pi$ is a Poisson bivector on the Lie algebroid $(A, \mu)$.

 We remark that $J_\pi \smc J_\pi=0$ and so, using (\ref{supergeometric_torsion}) with $\lambda=0$, we deduce that the torsion of $J_\pi$ is given by
${\mathcal T}_\mu J_\pi=\frac{1}{2} \{\pi, \{\pi, \mu\}\}$. Therefore,
${\mathcal T}_\mu J_\pi=0 \Leftrightarrow [\pi, \pi]_\mu=0$.

\

{\bf c)} Let $\omega$ be a $2$-form on $A$. Then,
$J_\omega = \left(
\begin{array}{cc}
 0 & 0 \\
\omega & 0
\end{array}
\right)$
is a deforming and a Nijenhuis tensor for the Courant algebroid $(A \oplus A^*, \mu)$.

This is an immediate consequence of  $J_\omega \smc J_\omega=0$ and $\mu_{J_{\omega}, J_{\omega}}=\{\omega, \{ \omega, \mu \} \}=0$.

\

{\bf d)} Let $N:A \to A$ be a $(1,1)$-tensor on  $A$, such that $N^2= \alpha \, id_A$, for some $\alpha \in \Rr$. Then, $I_N=
\left(
\begin{array}{cc}
 N & 0 \\
0 & -N^*
\end{array}
\right)$  is a Nijenhuis tensor for the Courant algebroid $(A \oplus A^*, \mu)$ if and only if $N$ is Nijenhuis tensor for the Lie  algebroid $(A, \mu)$ \cite{YKS11}.

\

{\bf e)} Let $\pi$ be a bivector on $A$
and $N:A \to A$  a $(1,1)$-tensor on  $A$. Then,
$J=\left(
\begin{array}{cc}
 N & \pi \\
0 & -N^*
\end{array}
\right)$ is deforming for $\mu$   if and only if
$
\left\{
 \begin{array}{l} N \, \mbox{\rm is deforming for}\, \mu\\
 \mu_{N,\pi}+\mu_{\pi,N}=0
 \\
 \pi \, \mbox{\rm is Poisson for} \, \mu.
  \end{array}
  \right.$

 We have,
 \begin{align*}
\mu_{{J,J}} & = \{N+\pi, \{N+\pi, \mu \} \} \\
& =\{N, \{N, \mu \} \}+ \{ \pi, \{N, \mu \} \}+ \{N, \{ \pi, \mu \} \}+ \{ \pi, \{ \pi, \mu \} \}\\
&= \mu_{N,N}+\mu_{N,\pi}+\mu_{\pi,N}+ \mu_{\pi,\pi}
\end{align*}
and, by counting the bidegrees, we deduce that
$\mu_{J,J}= \eta \, \mu$ if and only if
\begin{equation*}
\mu_{N,N}= \eta \, \mu, \, \, \, \, \,
 \mu_{N,\pi}+\mu_{\pi,N}=0
 , \,\,\, \, \,
[\pi, \pi ]_\mu=0.
\end{equation*}
\vspace*{-1cm}
\flushright{$\diamondsuit$}
 \end{ex}

\

 Let us consider the Courant algebroid $(A \oplus A^*, \mu + \gamma)$, which is the double of a Lie bialgebroid $((A, \mu), (A^*, \gamma))$ and the skew-symmetric  $(1,1)$-tensor $J: A \oplus A^* \to A \oplus A^*$, given by

\begin{equation}  \label{Jdef}
J= \left(
\begin{array}{cc}
\frac{1}{2}\,id_A & \pi \\
0 & -\frac{1}{2}\,id_{A^*}
\end{array}
\right).
\end{equation}

\begin{prop}
Let $((A,\mu), (A^*,\gamma))$ be a Lie bialgebroid. Then, the $(1,1)$-tensor $J$ given by (\ref{Jdef})  is a deforming tensor for the Courant structure $\mu + \gamma$ if and only if $\pi$ is a solution of the Maurer-Cartan equation
\begin{equation*}  \label{MC}
\d_\gamma \pi= \frac{1}{2} [\pi, \pi]_{\mu}.
\end{equation*}
\end{prop}
\begin{proof}
The $(1,1)$-tensor $J=\frac{1}{2}\,id_A+\pi$ is a deforming tensor for $ \mu + \gamma$ if there exists $\eta \in \Rr$ such that
\begin{equation*}
\left\{ \frac{1}{2}\,id_A +\pi, \left\{ \frac{1}{2}\,id_A +\pi, \mu + \gamma \right\} \right\} =\eta (\mu + \gamma).
\end{equation*}
We have,
\begin{eqnarray*}
 \lefteqn{\left\{\frac{1}{2}\,id_A +\pi, \left\{\frac{1}{2}\,id_A +\pi, \mu + \gamma \right\}\right\}  = \frac{1}{4} \{ \,id_A, \{ \,id_A, \mu \} + \{ \,id_A, \gamma \} \}} \\
& & + \frac{1}{2} \{ \,id_A, \{ \pi, \mu \} + \{ \pi, \gamma \} \}+
 \frac{1}{2}\{ \pi , \{
 \,id_A , \mu \} + \{ \,id_A , \gamma \} \}
+\{ \pi, \{ \pi, \mu \} + \{ \pi, \gamma \} \} \\
&=& \frac{1}{4} (\mu + \gamma) -2 \{\pi, \gamma \} - \{ \{ \pi, \mu \}, \pi \},
\end{eqnarray*}
where we used $\{id_A , u \} = (q-p) u$, for all $u$ of bidegree $(p,q)$.
So, $$\frac{1}{4} (\mu + \gamma) +2 \d_{\gamma} \pi - [\pi,\pi]_{\mu}= \eta  (\mu + \gamma)$$ if and only if
$$ \eta=\frac{1}{4} \quad {\rm and} \quad \d_\gamma \pi= \frac{1}{2} [\pi, \pi]_{\mu}.$$
\end{proof}


\section{Hierarchies of compatible tensors and structures}
We construct a hierarchy of compatible Courant structures on $(E,\langle .,. \rangle)$, that are obtained deforming an initial Courant structure by a Nijenhuis tensor. Then, we consider hierarchies of pairs of tensors which are compatible, in a certain sense, w.r.t. some deformed pre-Courant structures.

We introduce the following notation, where $I,J, \dots, K$ are skew-symmetric $(1,1)$-tensors on a pre-Courant algebroid $(E,\Theta)$:

\begin{itemize}
\item $\Theta_{I, J, \dots , T}= (((\Theta_I)_J)_{\dots})_T$;
\item $\Theta_k= \underbrace{ (((\Theta_I)_I)_{\dots})_I}_{k}= \Theta_{{\scriptsize \underbrace{I, \dots, I}_{k}}}$, $k \in \Nn$; $\Theta_0= \Theta$.
\end{itemize}

\subsection{Hierarchy of compatible Courant structures} \label{section:compatiblecourant}

In this section we construct a hierarchy of compatible Courant structures on $(E,\langle .,. \rangle)$.

The next proposition generalizes a result in \cite{magriYKS}.
\begin{prop} \label{torsions}
Let $I$ be a skew-symmetric $(1,1)$-tensor on a pre-Courant algebroid $(E,\Theta)$.  Then, for all sections $X$, $Y$ of $E$,
\begin{equation}  \label{torsion_theta}
{\mathcal T}_{{\Theta}_k}I (X,Y)= {\mathcal T}_{\Theta_{k-1}}I (IX,Y) + {\mathcal T}_{\Theta_{k-1}}I (X,IY) -I ({\mathcal T}_{\Theta_{k-1}}I (X,Y)), \, k \in \Nn.
\end{equation}
\end{prop}
\begin{proof}

Let us denote by $[.,.]_k$ the Dorfman bracket associated to $\Theta_k$. It is obvious that
$$[X,Y]_k=[IX,Y]_{k-1}+[X,IY]_{k-1}-I[X,Y]_{k-1},$$
and therefore we have,
\begin{align*}
{\mathcal T}_{{\Theta}_k}I (X,Y)& =[IX,IY]_{k} - I[IX,Y]_{k}-I[X,IY]_{k}+I^2[X,Y]_{k}\\
&= [I^2X,IY]_{k-1}-I[I^2X,Y]_{k-1}-I[IX,IY]_{k-1}+I^2[IX,Y]_{k-1}\\
& +[IX,I^2Y]_{k-1}-I[IX,IY]_{k-1}-I[X,I^2Y]_{k-1}+I^2[X,IY]_{k-1}\\
& -I([IX, IY]_{k-1}-I[IX,Y]_{k-1}-I[X,IY]_{k-1}+I^2[X,Y]_{k-1})\\
&= {\mathcal T}_{\Theta_{k-1}}I (IX,Y) + {\mathcal T}_{\Theta_{k-1}}I (X,IY) -I ({\mathcal T}_{\Theta_{k-1}}I (X,Y)).
\end{align*}
\end{proof}

\begin{cor} \label{Nijenhuis}
If $I$ is Nijenhuis for $\Theta$, then $I$ is Nijenhuis for $\Theta_k$, $\forall k \in \Nn$.
\end{cor}

When $(E,\Theta)$ is a Courant algebroid, it is well known \cite{grab} that if $I$
is a skew-symmetric Nijenhuis tensor for $\Theta$, then $(E,\Theta_I)$ is a Courant algebroid. Applying (\ref{torsion_theta}) we get, by recursion:

\begin{prop} \label{courant_hierar}
Let $(E,\Theta)$ be a Courant algebroid and $I$ a skew-symmetric Nijenhuis tensor for $\Theta$. Then, $(E,\Theta_k)$ is a Courant algebroid, for all $k \in \Nn$.
\end{prop}

\

We introduce the following notation: $I^n = \underbrace{I \smc \dots \smc I}_{n}$, for $n\geq 1$ and $I^0=id_E$.

Let us compute the torsion  ${\mathcal T}_{\Theta}I^n$, for all $n \in \Nn$.

\begin{prop}
Let $I$ be a $(1,1)$-tensor on a pre-Courant algebroid $(E,\Theta)$. Then, for all sections $X$ and $Y$ of $E$,
\begin{align*}
{\mathcal T}_{\Theta}I^n(X,Y)= & {\mathcal T}_{\Theta}I(I^{n-1}X,I^{n-1}Y)+ I({\mathcal T}_{\Theta}I^{n-1}(IX, Y)+ {\mathcal T}_{\Theta}I^{n-1}(X,I Y)) \nonumber \\ & -I^2({\mathcal T}_{\Theta}I^{n-2}(IX, IY))+ I^{2n-2}({\mathcal T}_{\Theta}I(X,Y)), \,\, n\geq 2.
\end{align*}
\end{prop}

\begin{proof}
Using the definition of Nijenhuis torsion we have, for all sections $X$ and $Y$ of $E$,

\begin{align}  \label{t1}
{\mathcal T}_{\Theta}I(I^{n-1}X,I^{n-1}Y)= & [I^n X,I^n Y]
-I([I^n X,I^{n-1}Y]+ [I^{n-1}X,I^{n}Y]) \nonumber \\
& +I^2 [I^{n-1}X,I^{n-1}Y];
\end{align}

\begin{align}  \label{t2}
I({\mathcal T}_{\Theta}I^{n-1}(IX,Y)+&{\mathcal T}_{\Theta} I^{n-1}(X,IY) ) = I([I^n X, I^{n-1}Y]+ [I^{n-1} X, I^{n}Y]) \nonumber \\
& -I^n([I^n X, Y]+ [I X, I^{n-1}Y]+ [I^{n-1} X, IY]+[X,I^nY])\nonumber\\ & + I^{2n-1}([IX,Y]+[X,IY]);
\end{align}

 \begin{align}  \label{t3}
-I^2 (  {\mathcal T}_{\Theta}I^{n-2}(IX,IY)) =&- I^2 [I^{n-1} X, I^{n-1}Y] +I^n( [I^{n-1} X, IY]+ [I X, I^{n-1}Y]) \nonumber \\
& -I^{2n-2} [IX,IY];
 \end{align}
 and
 \begin{equation}  \label{t4}
 I^{2n-2} (  {\mathcal T}_{\Theta}I(X,Y))= I^{2n-2}[IX,IY]-  I^{2n-1}([IX,Y]+[X,IY])
+ I^{2n}[X,Y].
 \end{equation}

 The sum of the right hand sides of equations (\ref{t1}), (\ref{t2}), (\ref{t3}) and (\ref{t4}) gives
 \begin{equation*}
 [I^nX,I^nY]-I^n([I^nX,Y]+ [X,I^nY])+I^{2n}[X,Y]={\mathcal T}_{\Theta}I^n(X,Y).
 \end{equation*}
\end{proof}

As an immediate consequence of the previous proposition, we have the following:

\begin{cor}\label{hierarchy_Nijenhuis}
If $I$ is a Nijenhuis tensor for $\Theta$, then $I^n$ is a Nijenhuis tensor for $\Theta$, for all $n \in \Nn$.
\end{cor}

\begin{prop}  \label{Nijenhuis_theta_m}
Let $I$ be a skew-symmetric $(1,1)$-tensor on a pre-Courant algebroid $(E,\Theta)$. If $I$ is Nijenhuis for $\Theta$, then $I^n$ is Nijenhuis for $\Theta_k$, for all $n,k \in \Nn$.
\end{prop}

\begin{proof}
Let $I$ be a Nijenhuis tensor for $\Theta$. Then, according to Corollary \ref{Nijenhuis}, $I$ is Nijenhuis for $\Theta_k$, for all $k \in \Nn$. Applying Corollary \ref{hierarchy_Nijenhuis}, the result follows.
\end{proof}

\

Recall that
two Courant structures $\Theta_1$ and $\Theta_2$ on a vector bundle $(E, \langle \, ,\, \rangle)$ are said to be {\em compatible} if their sum $\Theta_1 + \Theta_2$ is a Courant structure on $(E, \langle \, ,\, \rangle)$.
As an immediate consequence, we have that $\Theta_1$ and $\Theta_2$ are compatible if and only if
 \begin{equation*} \label{compatible}
\{\Theta_1 , \Theta_2 \}=0.
\end{equation*}
Two arbitrary pre-Courant structures $\Theta_1$ and $\Theta_2$ on $(E, \langle \, ,\, \rangle)$ are compatible, in the sense that the sum  $\Theta_1 + \Theta_2$ is always a pre-Courant structure.

\begin{thm}  \label{hierarchycourant}
Let $I$ be a skew-symmetric $(1,1)$-tensor on a Courant algebroid $(E,\Theta)$. If $I$ is Nijenhuis for $\Theta$,
 then the Courant structures $\Theta_k$ and $\Theta_m$ on $(E, \langle \, ,\, \rangle)$ are compatible, for all $k,m \in \Nn$.

In particular, $\Theta$ is compatible with $\Theta_k$, for all $k \in \Nn$.
\end{thm}

\begin{proof}
First, we remark that if $m=k$, then we have $\{\Theta_m, \Theta_m \}=0$ by Proposition \ref{courant_hierar}. Also,
for any Courant structure $\Theta$ and any skew-symmetric $(1,1)$-tensor $I$, the relation $\{\Theta, \Theta_I\}=0$ follows from the Jacobi
 identity and the graded symmetry of the Poisson bracket.
We use induction on $m+k$ to finish the proof.

\noindent Case $m+k=2$:

\begin{itemize}
\item i) $m=k=1$,
$$\{\Theta_I,\Theta_I \}=0;$$
\item ii) $m=2$, $k=0$,
$$\{\Theta_{I,I}, \Theta \}= \{I, \{\Theta,\Theta_I\}\}-\{\Theta_I,\Theta_I \}=0.$$
\end{itemize}

Now, suppose that $\{\Theta_m,\Theta_k\}=0$ holds with $m+k=s-1$ and take $m$ and $k$ such that $m+k=s$.
\begin{itemize}
\item i) if $m=k$, we already noticed that $\{\Theta_m,\Theta_m\}=0$;
\item ii) if $m \neq k$, suppose that $m>k$. Then,
\begin{align*}
\{\Theta_{m},\Theta_{k} \}& = \{\{I, \Theta_{m-1}\}, \Theta_k \}= \{I, \{\Theta_k, \Theta_{m-1}\}\}- \{\Theta_{k+1}, \Theta_{m-1} \}\\
&= -\{\Theta_{m-1},\Theta_{k+1} \} \\
&= -\{I, \{ \Theta_{m-2}, \Theta_{k+1} \}\}+ \{\Theta_{m-2},\Theta_{k+2} \}\\
&= \{\Theta_{m-2},\Theta_{k+2} \}.
\end{align*}

\noindent Applying the Jacobi identity successively, we get

\begin{align*}
\{ \Theta_m,\Theta_k \}&= \left\{
 \begin{array}{lll}
 (-1)^{m-l} \{ \Theta_l, \Theta_l \},& {\mbox {\rm if} \,\,  m+k=2l} \\
(-1)^{m-(l+1)}  \{ \Theta_{l+1}, \Theta_l \}, &  {\mbox {\rm if} \,\, m+k=2l+1}
\end{array}
 \right. \\
&= \left\{
 \begin{array}{lll}
0,& {\mbox {\rm if} \,\,  m+k=2l} \\
(-1)^{m-(l+1)} \frac{1}{2}\{I, \{ \Theta_{l}, \Theta_l \}\}=0, &  {\mbox {\rm if} \,\, m+k=2l+1}.
\end{array}
 \right.
\end{align*}
\end{itemize}
\end{proof}

\begin{rem}
The statements of Theorem \ref{hierarchycourant} still hold if we replace the assumption of $I$ being Nijenhuis for $\Theta$ by $I$ deforming for $\Theta$. In fact, if $\Theta_{I,I}=\eta \, \Theta$, for some $\eta \in \Rr$, then, a straightforward computation provides
$$\Theta_{2k}=\eta^k \Theta, \; \; \; \; \Theta_{2k+1}=\eta^k \Theta_I, \,\,\, {\hbox{\textrm for all}} \,\,\, k \in \Nn.$$
\end{rem}

We investigated so far the Courant structure $\Theta_n $ obtained by deforming $n$ times $\Theta$ by a skew-symmetric tensor $I$. It is logical to ask what happens when one deforms $\Theta$ by $I^n$.  The answer, which is given in the next proposition, is that we get precisely the same pre-Courant structure $\Theta_n $. However, this structure
can not be written as $\Theta_{I^n}$ for even $n$, since $I^n $ is not a skew-symmetric $(1,1)$-tensor. We bypass this difficulty by considering directly the Dorfman brackets, rather than the functions of degree $3$ associated with.

\begin{prop}  \label{lemA}
Let $(\rho, [.,.])$ be a pre-Courant structure on $(E, \langle.,.\rangle)$ and $I$  a  $(1,1)$-tensor on  $E$. Then, for all sections $X$ and $Y$ of $E$,
\begin{itemize}
\item[a)]
$
\displaystyle{[X,Y]_{I^{2n+1}}= [X,Y]_{I^{2n},I}-\sum_{\stackrel{
 0\leq i,j \leq 2n-1}{
{\scriptscriptstyle i+j=2n-1}}} I^j( {\mathcal T}_{\Theta}I(I^i X,Y)+  {\mathcal T}_{\Theta}I(X, I^i Y))};
$

\

\item[b)]
If $I$ is Nijenhuis for $(\rho, [.,.])$ then, for any $n \in \Nn$,

$[X,Y]_{I^{n}}=[X,Y]_{{{\scriptsize \underbrace{I, \dots,I}_{n}}}}$;

\

\item[c)]
If $I$ is Nijenhuis for $(\rho, [.,.])$ then, for any $m,n \in \Nn$,

$[X,Y]_{I^m,I^n}=[X,Y]_{I^{m+n}}.$
\end{itemize}
\end{prop}

\begin{proof}

\begin{itemize}
\item[a)] It is an easy but cumbersome computation.

\item[b)] First, we observe that, if two skew-symmetric $(1,1)$-tensors $I$ and $J$ commute, then $[X,Y]_{I,J}=[X,Y]_{J,I}$, for all  sections $X$ and $Y$ of $E$. In particular we have,
for all $m,n \in \Nn$,
\begin{equation}  \label{commute}
 [X,Y]_{I^m, I^n}=[X,Y]_{I^n, I^m}.
\end{equation}
\begin{itemize}
\item[i)] If $n$ is odd, $n=2k+1$, we use a):
$$[X,Y]_{I^n}=[X,Y]_{I^{2k+1}}=[X,Y]_{I^{2k},I },$$
and it is case ii).
\item[ii)] If  $n$ is even, $n=2k$, since $I^k$ is Nijenhuis, using (\ref{second_def_torsion}) we may write
    $$[X,Y]_{I^n}=[X,Y]_{I^k \smc I^k}=[X,Y]_{I^k, I^k}.$$
If $k$ is even, we repeat the procedure. If $k$ is odd, we are back to case i).

Repeating the procedure, and taking into account (\ref{commute}), we end up with
$$[X,Y]_{I^n}=[X,Y]_{{\scriptsize \underbrace{I, \dots, I}_{n}}},\,  \forall n \in \Nn.$$
 \end{itemize}
\item[c)] We use b) and (\ref{commute}):
$$[X,Y]_{I^n, I^m}= [X,Y]_{{\scriptsize \underbrace{I, \dots, I}_{n}}, I^m}=[X,Y]_{I^m,{\scriptsize \underbrace{I, \dots, I}_{n}}}=[X,Y]_{{\scriptsize \underbrace{I, \dots, I}_{m+n}}}=[X,Y]_{I^{m+n}}.$$
\end{itemize}

\end{proof}

\

Given a Courant structure $(\rho, [.,.])$ on $(E, \langle\, ,\, \rangle)$, we denote by $(\rho, [.,.])_I$ the pre-Courant structure on $(E, \langle\, ,\, \rangle)$ defined by $$(\rho, [.,.])_I:=( \rho \smc I, [.,.]_I),$$ where $I$ is a $(1,1)$-tensor on $E$. If $I$ is Nijenhuis for $(\rho, [.,.])$, then
\begin{equation}  \label{deformed_structure}
(\rho, [.,.])_{I^{k_1}, \cdots, I^{k_n}}=(\rho, [.,.])_{{\scriptsize \underbrace{I, \cdots,I}_{k_1+ \cdots+ k_n}}}=(\rho, [.,.])_{I^{k_1+ \cdots+ k_n}}
\end{equation}
is a Courant structure on $(E, \langle\, ,\, \rangle)$, for all $k_1, \cdots, k_n \in \Nn$, $n \in \Nn$. This result follows directly from Proposition \ref{lemA} b) and c). In supergeometric terms, (\ref{deformed_structure}) means that the deformation of $\Theta$, either by $I^{k_1+ \cdots+ k_n}$ or successively by $I^{k_1}, I^{k_2}, \cdots, I^{k_n}$, is the pre-Courant structure $\Theta_{{\scriptsize \underbrace{I, \cdots,I}_{k_1+ \cdots+ k_n}}}=\Theta_{k_1+ \cdots+ k_n}$.

\subsection{Hierarchy of compatible tensors w.r.t. $\Theta$}  \label{section2.2}
In this section, we introduce the notion of compatible pair of $(1,1)$-tensors with respect to a pre-Courant structure $\Theta$ on $E$ and construct a hierarchy of pairs of tensors satisfying this type of compatibility.

The notion of concomitant of two $(1,1)$-tensors on a manifold was introduced in \cite{magrimorosi} and then extended to Lie algebroids in \cite{magriYKS}. For pre-Courant algebroids it can be defined as follows:

\begin{defn}
The concomitant of two skew-symmetric $(1,1)$-tensors $I$ and $J$ on a pre-Courant algebroid $(E,\Theta)$ is given by
\begin{equation}  \label{def_conc}
C_\Theta(I,J)=\{J, \{I, \Theta \}\}+\{I, \{J, \Theta \}\}=\Theta _{I,J}+\Theta_{J,I}.
\end{equation}
\end{defn}
\noindent Using the Jacobi identity, we easily check that (\ref{def_conc}) is equivalent to
\begin{equation*}
C_\Theta(I,J)=\Theta_{\{J,I\}}+2 \Theta_{J,I}.
\end{equation*}
If $(\rho, [.,.])$ is the pre-Courant structure on $E$ corresponding to $\Theta$, (\ref{def_conc}) reads as follows:
\begin{equation} \label{def_conc1}
 \{\{X, C_\Theta(I,J) \},Y \}= [X,Y]_{I,J}+ [X,Y]_{J,I}
 \end{equation}
 and
 \begin{equation*}
 \{\{X, C_\Theta(I,J) \},f \}= (\rho \smc (I \smc J+J\smc I))(X).f,
  \end{equation*}
for all $X,Y \in \Gamma(E)$ and $f \in C^\infty(M)$.

In the sequel, we denote the left hand side of (\ref{def_conc1}) by $C_\Theta(I,J)(X,Y)$. When $I$ and $J$ anti-commute, we have
$$ \{\{X, C_\Theta(I,J) \},f \}=0,$$
for all  $X \in \Gamma(E)$ and $f \in C^\infty(M)$. Therefore, in this case,
\begin{equation}  \label{conc_0}
C_\Theta(I,J)=0 \Leftrightarrow C_\Theta(I,J)(X,Y)=0, \, \, \forall X,Y \in \Gamma(E).
\end{equation}

\begin{defn}
We say that two skew-symmetric $(1,1)$-tensors $I$ and $J$ on a pre-Courant algebroid $(E,\Theta)$
 {\em anti-commute with respect to} $\Theta$, if $\Theta_{I,J}= -\Theta_{J,I}$ or, equivalently, if $C_\Theta(I,J)=0$.
\end{defn}

For the various classes of pairs of skew-symmetric $(1,1)$-tensors that will be introduced in the sequel, we shall require that the skew-symmetric $(1,1)$-tensors are compatible in the following sense:

\begin{defn}  \label{def_compatible_pair}
A pair $(I,J)$ of
skew-symmetric $(1,1)$-tensors on a pre-Courant algebroid $(E,\Theta)$ is said to be a {\em compatible pair w.r.t.} $\Theta$, if $I$ and $J$ anti-commute and anti-commute w.r.t. $\Theta$.
\end{defn}

Let $I$ and $J$ be two $(1,1)$-tensors on a pre-Courant algebroid $(E, \Theta)$. Recall that the {\em Nijenhuis concomitant} of $I$ and $J$ is defined, for all sections $X$ and $Y$ of $E$, as follows \cite{kobay}:
\begin{eqnarray} \label{torsion_sum}
{\mathcal{N}}_\Theta(I,J)(X,Y)&=& [IX,JY]-I[X,JY]-J[IX,Y]+IJ[X,Y] \nonumber \\
& &+[JX,IY]-J[X,IY]-I[JX,Y]+JI[X,Y].
\end{eqnarray}
Notice that, if $I=J$, then ${\mathcal N}_\Theta(I,I)=2 {\mathcal T}_\Theta I$ and, if $I$ and $J$ anti-commute, then ${\mathcal N}_\Theta(I,J)=\frac{1}{2} C_\Theta(I,J)$.

\begin{lem}  \label{formula torsion_sum}
Let $I$ and $J$ be two skew-symmetric $(1,1)$-tensors on a pre-Courant algebroid $(E,\Theta)$. Then, ${\mathcal T}_\Theta(I+J)= {\mathcal T}_\Theta I+{\mathcal T}_\Theta J + {\mathcal{N}}_\Theta(I,J)$.
\end{lem}

\begin{proof}
Let $I$ and $J$ be two skew-symmetric $(1,1)$-tensors on  $(E,\Theta)$ and $X, Y$ any sections of $E$. Then, using the definition of Nijenhuis torsion, we get:
\begin{align*}
{\mathcal T}_\Theta(I+J)(X,Y)=& {\mathcal T}_\Theta I(X,Y)+ {\mathcal T}_\Theta J(X,Y) + [IX,JY]+[JX,IY]-I[X,JY] \\
& -J[X,IY]-I[JX,Y]-J[IX,Y]+ IJ[X,Y]+JI[X,Y]\\
=&{\mathcal T}_\Theta I(X,Y)+ {\mathcal T}_\Theta J(X,Y)+{\mathcal{N}}_\Theta(I,J)(X,Y).
\end{align*}
\end{proof}

The next proposition gives a characterization of compatible pairs.

\begin{prop}
Let $I$ and $J$ be two anti-commuting skew-symmetric $(1,1)$-tensors on a pre-Courant algebroid $(E,\Theta)$. Then, $(I,J)$ is a compatible pair w.r.t. $\Theta$ if and only if ${\mathcal T}_\Theta(I+J)= {\mathcal T}_\Theta I+{\mathcal T}_\Theta J$.
\end{prop}

\begin{prop} \label{extensionmagri}
Let  $I$ and $J$ be two anti-commuting skew-symmetric $(1,1)$-tensors on a pre-Courant algebroid $(E,\Theta)$. Then, for all sections $X, Y$ of $E$ and $n \geq 1$,
 \begin{align} \label{concomitant_recursion}
C_\Theta(I,I^{n} \smc J)(X,Y)= & I( C_\Theta(I,I^{n-1} \smc J)(X,Y))   \nonumber  \\
& + 2\,{\mathcal T}_{\Theta}I ((I^{n-1} \smc J)X,Y)+ 2\,{\mathcal T}_{\Theta}I(X, (I^{n-1} \smc J)Y).
 \end{align}
\end{prop}

\begin{proof}

A simple computation gives
\begin{align*}
C_\Theta(I,I^{n} \smc J)(X,Y)= &[X,Y]_{I,I^n \smc J}+[X,Y]_{I^n \smc J,I} \\
&= 2\left([I^n(JX), IX]- I[I^n(JX), Y]+[IX, I^n(JY)] \right.\\
&  \left. -I[X, I^n(JY)]-I^n \smc J[IX,Y]-I^n \smc J[X,IY] \right),
\end{align*}
for all sections $X,Y$ of $E$ and $n \geq 1$.
Thus, we have
\begin{align*}
I( C_\Theta(I,I^{n-1} \smc J)(X,Y))= 2\left(I[I^{n-1}(JX), IY]-I^2[I^{n-1}(JX),Y] \right. \\
 \left. +I[IX, I^{n-1}(JY)]-I^2[X, I^{n-1}(JY)]-I^n \smc J [IX,Y]-I^n \smc J[X,IY] \right).
\end{align*}
Since
$${\mathcal T}_\Theta I (I^{n-1} (JX),Y)=[I^n (JX),IY]-I([I^n  (JX),Y]+[I^{n-1}(JX),IY]- I[I^{n-1} (JX),Y])$$
and
$${\mathcal T}_\Theta I (X,I^{n-1} (JY))=[IX,I^n (JY)]-I([IX,I^{n-1}  (JY)]+[X,I^{n}(JY)]- I[X,I^{n-1} (JY)]),$$
the result follows.
\end{proof}

\begin{thm}  \label{propconcIn}
Let  $I$ and $J$ be two skew-symmetric $(1,1)$-tensors on a pre-Courant algebroid $(E,\Theta)$  such that ${\mathcal T}_{\Theta}I (JX,Y)={\mathcal T}_{\Theta}I(X,JY)=0$, for all sections $X$ and $Y$ of $E$. If $(I,J)$ is a compatible pair w.r.t. $\Theta$, then
 \begin{equation}  \label{concIn}
 C_\Theta(I, I^n \smc J)=0,
 \end{equation}
 and $(I, I^n \smc J)$ is a compatible pair w.r.t. $\Theta$, for all $n \in \Nn$.
 \end{thm}

\begin{proof}
For $n=0$, (\ref{concIn}) reduces to $C_\Theta(I,J)=0$ and $(I,J)$ is a compatible pair w.r.t. $\Theta$, which is one of the assumptions. From
 (\ref{concomitant_recursion}), we get
$$ C_\Theta(I,I^{n} \smc J)(X,Y)=  I( C_\Theta(I,I^{n-1} \smc J)(X,Y)), \,\, n\geq 1,$$
for all sections $X,Y$ of $E$, where we used $I^{n-1} \smc J = (-1)^{n-1} J \smc  I^{n-1}$ to obtain $${\mathcal T}_{\Theta}I ((I^{n-1} (JX),Y)=(-1)^{n-1} {\mathcal T}_{\Theta}I(J(I^{n-1} X),Y)=0$$ and analogously $${\mathcal T}_{\Theta}I(X, I^{n-1}(JY))=0.$$
Therefore, using (\ref{conc_0}),
 it is obvious that if $C_\Theta(I, I^{n-1} \smc J)=0$, then $C_\Theta(I,I^{n} \smc J)=0$ and (\ref{concIn}) follows by recursion.
 Since $I$ anti-commutes with $I^{n} \smc J$, the proof is complete.
\end{proof}


\subsection{Compatible tensors w.r.t. $\Theta_k$, $k \in \Nn$}  \label{section2.3}
In this section, we address the general case of hierarchies of tensors which are compatible w.r.t. each term of the family  $(\Theta_k)_{k \in \Nn}$ of pre-Courant structures on $E$.

\begin{prop}\label{propconcThetaI}
Let  $I$ and $J$ be two skew-symmetric $(1,1)$-tensors on a pre-Courant algebroid $(E,\Theta)$. Then,
\begin{equation*}\label{C_theta1}
C_{\Theta_{I}}(I,J)= C_{\Theta}(I, \{J,I \})+ \{I, C_{\Theta}(I,J)\}.
\end{equation*}
In particular, if $I$ and $J$ anti-commute, then,
\begin{equation}\label{C_theta2}
C_{\Theta_{I}}(I,J)= 2 C_{\Theta}(I, I \smc J)+ \{I, C_{\Theta}(I,J)\}.
\end{equation}
\end{prop}

\begin{proof}
Applying twice the Jacobi identity, we get
\begin{align*}
    \Theta_{I,I,J}&=\Theta_{I,\{J,I\}} + \Theta_{I,J,I}\\
    &=\Theta_{I,\{J,I\}} + \Theta_{\{J,I\},I} + \Theta_{J,I,I},
\end{align*}
which can be written as $$C_{\Theta}(I, \{J,I\})=\Theta_{I,I,J}-\Theta_{J,I,I}.$$

From the definition of $C_{\Theta}(I,J)$, we have $\Theta_{J,I,I}=\{I, C_{\Theta}(I,J)\}- \Theta_{I,J,I}$. Substituting this result in the last equality, we get
\begin{align*}
    C_{\Theta}(I, \{J,I\})&=\Theta_{I,I,J}-\{I, C_{\Theta}(I,J)\} + \Theta_{I,J,I}\\
    &=C_{\Theta_I}(I,J) - \{I, C_{\Theta}(I,J)\},
\end{align*}
proving the first statement. If $I$ and $J$ anti-commute, then $\{J,I\}=2 I \smc J$ and the second statement follows.
\end{proof}

The next theorem extends the result of Theorem~\ref{propconcIn}.

\begin{thm}  \label{new thma}
Let  $I$ and $J$ be two skew-symmetric $(1,1)$-tensors on a pre-Courant algebroid $(E,\Theta)$  such that ${\mathcal T}_{\Theta}I (JX,Y)={\mathcal T}_{\Theta}I(X,JY)=0$, for all sections $X$ and $Y$ of $E$. If $(I,J)$ is a compatible pair w.r.t. $\Theta$, then
$C_{\Theta_k}(I,I^n\smc J)=0$ and
$(I,I^n\smc J)$ is a compatible pair w.r.t. $\Theta_k$, for all $k,n \in \Nn$.
\end{thm}

\begin{proof}
Let $I$ and $J$ be two skew-symmetric $(1,1)$-tensors which are compatible w.r.t $\Theta$.
Suppose that ${\mathcal T}_{\Theta}I(JX,Y)= {\mathcal T}_{\Theta}I(X, JY)=0$, for all sections $X$ and $Y$ of $E$. We will prove, by induction on $k$, that
        $$C_{\Theta_k}(I,I^n \smc  J)=0, \quad \forall k,n \in \mathbb{N}.$$
        For $k=0$, this is the content of Theorem~\ref{propconcIn}.

        Suppose now that, for some $k \in \mathbb{N}$, $C_{\Theta_k}(I,I^n \smc  J)=0$,  for all $n \in \mathbb{N}$. Then, from (\ref{C_theta2})  we have, for all $n \in \mathbb{N}$,
        \begin{equation*}
            C_{\Theta_{k+1}}(I,I^n \smc  J)=2 C_{\Theta_k}(I,I^{n+1} \smc J) + \{I, C_{\Theta_k}(I,I^n \smc J)\}=0,
        \end{equation*}
        where we used the induction hypothesis in the last equality. Since the skew-symmetric tensor $I^n \smc J$ anti-commutes with $I$, for all $n \in \Nn$, $(I,I^n\smc J)$ is a compatible pair w.r.t. $\Theta_k$, for all $k,n \in \Nn$.
         \end{proof}

In order to establish the main results of this section, we need the following lemmas.

\begin{lem} \label{IoJ}
Let  $I$ and $J$ be two anti-commuting skew-symmetric $(1,1)$-tensors on a pre-Courant algebroid $(E,\Theta)$. Then,
 \begin{equation*}
 C_\Theta(I,J)=2(\Theta _{I,J}- \Theta_{I \smc J}).
 \end{equation*}
 \end{lem}
 \begin{proof}
Since $I$ and $J$ anti-commute,  $\{I,J\}=-2 \,I \smc J$. Using the Jacobi identity, we have
$
\Theta_{J,I}= -2 \Theta_{I \smc J} + \Theta_{I,J}.$
Therefore,
$ C_\Theta(I,J)=\Theta _{I,J}+\Theta_{J,I}=2(\Theta _{I,J}-\Theta_{I \smc J}).$
\end{proof}

\begin{lem}  \label{lem20}
Let $I$ and $J$ be two skew-symmetric $(1,1)$-tensors on a pre-Courant algebroid $(E,\Theta)$ such that $I$ is Nijenhuis for $\Theta$. If $(I,J)$ is a compatible pair w.r.t. $\Theta$, then, for all sections $X$ and $Y$ of $E$,
\begin{equation*}
[X,Y]_{{{\scriptsize \underbrace{I, \dots,I}_{n}},J}}=[X,Y]_{I^n \smc J}.
\end{equation*}
\end{lem}

\begin{proof}
Theorem \ref{new thma} ensures that, for all $n \in \Nn$, $C_{\Theta_n}(I,J)=0$ and, applying Lemma \ref{IoJ} for the pre-Courant structure $\Theta_{n-1}$, we get
\begin{align*}
[X,Y]_{{{\scriptsize \underbrace{I, \dots,I}_{n}},J}}& =[X,Y]_{{{\scriptsize \underbrace{I, \dots,I}_{n-1}},I,J}}=\{\{X, (\Theta_{n-1})_{I,J}\}, Y\}\\
&= \{\{X, (\Theta_{n-1})_{I \smc  J}\}, Y\}=
[X,Y]_{{{\scriptsize \underbrace{I, \dots,I}_{n-1}},I \smc J}}.
\end{align*}

Since, for every $k \in \Nn$, $I$ anti-commutes with $I^k \smc J$, we may repeat $n-1$ times this procedure to end up with
$$[X,Y]_{{{\scriptsize \underbrace{I, \dots,I}_{n}},J}}=[X,Y]_{I^n \smc J}.$$
\end{proof}

\begin{rem}
In Lemma \ref{lem20}, we may replace the assumption that $I$ is Nijenhuis for $\Theta$ by ${\mathcal T}_{\Theta}I(JX,Y)= {\mathcal T}_{\Theta}I(X, JY)=0$, for all sections $X$ and $Y$ of $E$.
\end{rem}

\begin{thm}\label{propconcgeral}
    Let $I$ and $J$ be two skew-symmetric $(1,1)$-tensors on a pre-Courant algebroid $(E,\Theta)$, such that $I$ is Nijenhuis and $(I,J)$ is a compatible pair  w.r.t. $\Theta$. Then,
     \begin{equation}  \label{item_b}
  C_{\Theta_k}(I^{2s+1},I^n\smc J)=0
  \end{equation}
and
$(I^{2s+1},I^n\smc J)$ is a compatible pair w.r.t. $\Theta_k$, for all $k,n,s \in \mathbb{N}$.
      Moreover, if $J$ is Nijenhuis tensor, then
  \begin{equation}  \label{item_c}
         C_{\Theta_k}(I^{2s+1},I^n\smc J^{2m+1})=0
         \end{equation}
and
      $(I^{2s+1},I^n \smc J^{2m+1})$ is a compatible pair w.r.t. $\Theta_k$, for all $k,m,n,s \in \mathbb{N}$.
\end{thm}
\begin{proof}
Let $I$ and $J$ be two skew-symmetric $(1,1)$-tensors which are compatible w.r.t $\Theta$ and
such that ${\mathcal T}_{\Theta}I=0$. Firstly, we prove that
        $$C_{\Theta}(I^{2s+1},I^n \smc J)=0, \quad \forall s,n \in \mathbb{N}.$$
        Since $I^{2s+1}$ anti-commutes with ${I^n \smc J}$, we may apply Lemma \ref{IoJ}:
        $$C_\Theta(I^{2s+1},{I^n \smc J})(X,Y)=2 ([X,Y]_{I^{2s+1},{I^n \smc J}}- [X,Y]_{I^{2s+1}\smc ({I^n \smc J})}).$$
From Theorem~\ref{propconcIn}, $(I, I^n \smc J)$ is a compatible pair w.r.t. $\Theta$ and, applying
Lemma \ref{lem20}, we get
\begin{align*}
            C_\Theta(I^{2s+1},{I^n \smc  J})(X,Y)&= 2([X,Y]_{I^{2s+1},{I^n \smc J}}- [X,Y]_{{\scriptsize \underbrace{I, \dots, I}_{2s+1},{I^n \smc J}}})\\
            &=2([X,Y]_{I^{2s+1}}-[X,Y]_{I^{2s+1}})_{I^n \smc J}=0,
        \end{align*}
        where we have used Lemma \ref{lemA} b) in the second equality. From (\ref{conc_0}), we get $C_\Theta(I^{2s+1},{I^n \smc  J})=0$.

        In order to prove the result for a general $\Theta_k$, notice that, due to Corollary~\ref{Nijenhuis} and Theorem~\ref{new thma}, the hypothesis originally satisfied for $\Theta$, are also satisfied for any of the pre-Courant structures $\Theta_k, k\in \mathbb{N}$. Therefore, we can replace in the above arguments $\Theta$ by any $\Theta_k, k\in \mathbb{N}$.

Now, suppose that $I$ and $J$ are both Nijenhuis for $\Theta$. Since they play symmetric roles, we may intertwine them in (\ref{item_b}) and, taking  $k=0$, $n=0$ and $s=m$, we obtain $C_\Theta (I, J^{2m+1})=0$. Because $I$ and $J^{2m+1}$ anti-commute, we conclude that $(I,J^{2m+1})$ is a compatible pair w.r.t. $\Theta$. Thus, we may apply again (\ref{item_b}), taking $J^{2m+1}$ instead of $J$, to get $C_{\Theta_k}(I^{2s+1},I^n \smc J^{2m+1})=0$ and, because $I^{2s+1}$ anti-commutes with $I^n\smc J^{2m+1}$, the pair $(I^{2s+1},I^n\smc J^{2m+1})$ is a compatible pair w.r.t. $\Theta_k$, for all $k,m,n,s \in \mathbb{N}$.
\end{proof}


\section{Hierarchies of deforming-Nijenhuis pairs}

We introduce the notion of deforming-Nijenhuis pair as well as the definitions of Poisson tensor and Poisson-Nijenhuis pair on a pre-Courant algebroid. We construct several hierarchies of deforming-Nijenhuis and Poisson-Nijenhuis pairs.

\subsection{Hierarchy of deforming-Nijenhuis pairs for $\Theta_k$, $k \in \Nn$}
Starting with a deforming-Nijenhuis pair $(J,I)$ for $\Theta$, we prove, in a first step, that it is also a deforming-Nijenhuis pair for  $\Theta_k$, for all $k \in \Nn$. Then, we
construct a hierarchy $(J, I^{2n+1})_{n \in \Nn}$ of deforming-Nijenhuis pairs for $\Theta_k$, for all $k \in \Nn$.

\begin{defn}
Let $I$ and $J$ be two skew-symmetric $(1,1)$-tensors on a pre-Courant algebroid $(E, \Theta)$. The pair $(J,I)$ is said to be a {\em deforming-Nijenhuis pair} for $\Theta$ if
\begin{itemize}
\item $(J,I)$ is a compatible pair  w.r.t. $\Theta$;
\item $J$ is deforming for $\Theta$;
\item $I$ is Nijenhuis for $\Theta$.
\end{itemize}
\end{defn}

We need the following lemmas.

\begin{lem} \label{theta_cocycle}
Let $I$ and $J$ be two anti-commuting skew-symmetric $(1,1)$-tensors on a pre-Courant algebroid $(E,\Theta)$. Then, for all $k \in \Nn$,
\begin{equation}  \label{eq29}
 \left((\Theta_r)_{\{J, \{I,J\}\}} \right)_{\scriptsize \underbrace{I, \dots, I}_s}=(\Theta_{\{J, \{I,J\}\}})_{\scriptsize \underbrace{I, \dots, I}_k},
\end{equation}
for all $r, s \in \Nn$ such that $r+s=k$.

In particular,
\begin{itemize}
\item [i)] if $\Theta_{\{J, \{I,J\}\}}=\lambda_0 \Theta_{J,J,I}$, for some $\lambda_0 \in \Rr$, then
\begin{equation*}
(\Theta_k)_{\{J, \{I,J\}\}}=\lambda_0 (\Theta_{J,J})_{ {\scriptsize \underbrace{I, \dots, I}_{k+1}}}, \, \, \forall  k \in \Nn;
\end{equation*}
\item [ii)] if $\{J, \{I,J\}\}$ is a $\Theta$-cocycle, then it is a $\Theta_k$-cocycle, for all $k \in \Nn$.
\end{itemize}
\end{lem}

\begin{proof}
Let $I$ and $J$ be two skew-symmetric $(1,1)$-tensors on $(E,\Theta)$ that anti-commute. Then, we have
\begin{equation}  \label{commuting}
I \smc (I \smc J^2)=(I \smc J^2) \smc I \Leftrightarrow \{I, I \smc J^2 \}=0 \Leftrightarrow \{I, \{J, \{J, I \} \} \}=0.
\end{equation}
Using the Jacobi identity, (\ref{commuting}) implies
\begin{equation}  \label{theta_commuting}
\Theta_{I, \{J, \{J, I \} \}}= \Theta_{\{J, \{J, I \} \},I}.
\end{equation}
Since (\ref{theta_commuting}) holds for any pre-Courant structure on $E$, we may write
\begin{equation*}
(\Theta_{\scriptsize \underbrace{I, \dots, I}_{r+s}})_{\{J, \{I,J\}\}}=(\Theta_{\scriptsize \underbrace{I, \dots, I}_{r+s-1}})_{\{J, \{I,J\}\},I}.
\end{equation*}
Repeating the procedure $(s-1)$ times, we get (\ref{eq29}). The particular cases follow immediately.
 \end{proof}

 \begin{lem}  \label{lemJdeforming}
Let $I$ and $J$ be two skew-symmetric $(1,1)$-tensors on a pre-Courant algebroid $(E,\Theta)$. Then,
\begin{equation}  \label{thetaJIJ}
\Theta_{J,I,J}= \frac{1}{3} \left( \Theta_{J,J,I}+ \Theta_{\{J,\{I,J\}\}}+ \{J, C_{\Theta} (I,J)\} \right);
\end{equation}
\begin{equation}  \label{thetaIJJ}
\Theta_{I,J,J}= -\frac{1}{3} \left( \Theta_{J,J,I}+ \Theta_{\{J,\{I,J\}\}}-2\{J, C_{\Theta} (I,J)\} \right).
\end{equation}
\end{lem}

\begin{proof}
The formul{\ae} are obtained by application of the Jacobi identity.
 \end{proof}

\

As a particular case of the previous lemma, we have the following:

\begin{cor}
If $I$ and $J$ anti-commute w.r.t. $\Theta$, and if $\Theta_{\{J,\{I,J\}\}}=\lambda_0\Theta_{J,J,I}$, $\lambda_0 \in \Rr$, then
\begin{equation}  \label{alpha}
\Theta_{I,J,J}= \alpha \Theta_{J,J,I},
\end{equation}
with $\alpha= - \frac{\lambda_0+1}{3}$. Moreover, if $J$ is deforming for $\Theta$, i.e., $\Theta_{J,J}=\eta \, \Theta$, with $\eta  \in \Rr$, then  $J$ is deforming for $\Theta_I$. More precisely,
$
\Theta_{I,J,J}= \eta  \, \alpha \, \Theta_I
$.
\end{cor}

\

\begin{lem}  \label{theta_n_JJI}
Let $I$ and $J$ be two skew-symmetric $(1,1)$-tensors on a pre-Courant algebroid $(E,\Theta)$ such that $(I,J)$ is a compatible pair w.r.t. $\Theta$ and ${\mathcal T}_\Theta I(JX,Y)= {\mathcal T}_\Theta I (X,JY)=0$, for all sections $X$ and $Y$ of $E$. If $\Theta_{\{J,\{I,J\}\}}=\lambda_0\Theta_{J,J,I}$, for some $\lambda_0 \in \Rr \backslash \{\frac{4}{(-3)^m -1}, m \in \Nn \}$, then, for all $k \in \Nn$,
\begin{enumerate}
  \item[(a)] $(\Theta_{k})_{\{J,\{I,J\}\}}=\lambda_k(\Theta_k)_{J,J,I}$, where $\lambda_k$ is defined recurrently\footnote{Explicitly, $\lambda_k=\frac{(-3)^k \lambda_0}{1+ \frac{1-(-3)^k}{4}\lambda_0}$, for all $k \in \mathbb{N}$.} by $\lambda_k=\frac{-3\lambda_{k-1}}{1+\lambda_{k-1}}$, $k \geq 1$,
  \item[(b)] $\lambda_k (\Theta_k)_{J,J,I}= \lambda_0 \Theta_{J,J, {\scriptsize \underbrace{I, \dots, I}_{k+1}}}$.
  \item[(c)] If, in particular, $\lambda_0=0$, then $(\Theta_k)_{J,J}=(-\frac{1}{3})^k \Theta_{J,J, {\scriptsize \underbrace{I, \dots, I}_{k}}}$, for all $k \in \Nn$.
\end{enumerate}
\end{lem}

\begin{proof}
\begin{enumerate}
  \item[(a)] We will prove this statement by induction. Suppose that, for some $k \geq 1$, $(\Theta_{k-1})_{\{J,\{I,J\}\}}=\lambda_{k-1}(\Theta_{k-1})_{J,J,I}$. Using Lemma~\ref{theta_cocycle} and the induction hypothesis, we have
      \begin{align*}
        (\Theta_{k})_{\{J,\{I,J\}\}}&=(\Theta_{k-1})_{\{J,\{I,J\}\},I}\\
        &=\lambda_{k-1}(\Theta_{k-1})_{J,J,I,I}.
      \end{align*}
      Applying formula (\ref{alpha}) for $\Theta_{k-1}$, we obtain
      $$(\Theta_{k})_{\{J,\{I,J\}\}}=\frac{-3\lambda_{k-1}}{1+\lambda_{k-1}}(\Theta_{k-1})_{I,J,J,I}
      =\lambda_k(\Theta_{k})_{J,J,I},$$
      with $\lambda_k=\frac{-3\lambda_{k-1}}{1+\lambda_{k-1}}$.
  \item[(b)] Starting from the previous statement, then using the Lemma~\ref{theta_cocycle} and the hypothesis, we have,
     $$
        \lambda_k (\Theta_{k})_{J,J,I}=(\Theta_{k})_{\{J,\{I,J\}\}}=\Theta_{\{J,\{I,J\}\},  {\scriptsize \underbrace{I, \dots, I}_{k}}}=\lambda_0\Theta_{J,J, {\scriptsize \underbrace{I, \dots, I}_{k+1}}}.
     $$
      \item[(c)] From Lemma~\ref{theta_cocycle} i), we get
      $$(\Theta_{k})_{\{J, \{I,J\}\}}=0, \,\,\forall k \in \Nn,$$
      while Theorem~\ref{new thma} gives
      $$C_{\Theta_k}(I,J)=0, \,\,\forall k \in \Nn.$$
      Thus, applying successively the formula (\ref{thetaIJJ}), yields
    $$
      (\Theta_k)_{J,J}= -\frac{1}{3} (\Theta_{k-1})_{J,J,I}= \cdots = (-\frac{1}{3})^k \Theta_{J,J, {\scriptsize \underbrace{I, \dots, I}_{k}}}.
      $$
\end{enumerate}
\end{proof}

 \begin{prop} \label{Jdeforming_n}
Let $I$ and $J$ be two skew-symmetric $(1,1)$-tensors on a pre-Courant algebroid $(E,\Theta)$ such that $(I,J)$ is a compatible pair w.r.t.  $\Theta$,  $ \Theta_{\{J, \{I,J\}\}}=\lambda_0 \Theta_{J,J,I}$, for some $\lambda_0 \in \Rr \backslash \{\frac{4}{(-3)^m -1}, m \in \Nn \}$ and ${\mathcal T}_\Theta I(JX,Y)={\mathcal T}_\Theta I(X,JY)=0$, for all sections $X$ and $Y$ of $E$. If $J$ is a deforming tensor for $\Theta$, then $J$ is also a deforming tensor for $\Theta_k$, for all $k \in \Nn$.
\end{prop}
\begin{proof}
We consider two cases, depending on the value of $\lambda_0$.
\begin{itemize}
\item[i)] Case $\lambda_0 \neq 0$.

From Theorem \ref{new thma}, we  have that $C_{\Theta_{k}}(I,J)=0$, for all $k \in \Nn$.
We compute\footnote{Notice that if $\lambda_0 \neq 0$ then $\lambda_k \neq 0, \forall k \in \mathbb{N}$.}, using Lemma \ref{lemJdeforming} and both statements of Lemma \ref{theta_n_JJI}:
\begin{align*}
(\Theta_k)_{J,J} =(\Theta_{k-1})_{I,J,J} &= -\frac{1}{3}((\Theta_{k-1})_{J,J,I}+ (\Theta_{k-1})_{\{J,\{I,J\}\}})\\
&= -\frac{1}{3}((\Theta_{k-1})_{J,J,I}+ \lambda_{k-1}(\Theta_{k-1})_{J,J,I})\\
&= -\frac{1+\lambda_{k-1}}{3}(\Theta_{k-1})_{J,J,I}\\
&= -\frac{(1+\lambda_{k-1})\lambda_0}{3 \lambda_{k-1}} \Theta_{J,J, {\scriptsize \underbrace{I, \dots, I}_{k}}}\\
&= \frac{\lambda_0}{\lambda_k}\Theta_{J,J, {\scriptsize \underbrace{I, \dots, I}_{k}}}.
\end{align*}

The tensor $J$ being deforming for $\Theta$, we have $\Theta_{J,J}= \eta \, \Theta$, for some $\eta \in \Rr$, and the last equality becomes
$$(\Theta_k)_{J,J}= \frac{\lambda_0}{\lambda_k}\, \eta\, \Theta_k,$$
which means that $J$ is a deforming tensor for $\Theta_k$.

\item[(ii)] Case $\lambda_0 = 0$.

If $J$ is deforming for $\Theta$, i.e., $\Theta_{J,J}=\eta \, \Theta$, with $\eta \in \Rr$, then, from Lemma~\ref{theta_n_JJI} c) we immediately get
$$(\Theta_k)_{J,J}=(-\frac{1}{3})^k \eta \, \Theta_k, \,\,\, \forall k\in \Nn,$$
which means that $J$ is deforming for $\Theta_k$.
\end{itemize}
\end{proof}

Now, we establish the main result of this section.

\begin{thm}  \label{oddpowerI}
Let $I$ and $J$ be two skew-symmetric $(1,1)$-tensors on a pre-Courant (respectively, Courant) algebroid $(E,\Theta)$ such that $ \Theta_{\{J, \{I,J\}\}}=\lambda_0 \Theta_{J,J,I}$, for some $\lambda_0 \in \Rr \backslash \{\frac{4}{(-3)^m -1}, m \in \Nn \}$. If $(J,I)$ is a deforming-Nijenhuis pair for $\Theta$, then $(J, I^{2n+1})$ is a deforming-Nijenhuis pair for the pre-Courant (respectively, Courant) structures $\Theta_k$, for all $k,n \in \Nn$.
\end{thm}

\begin{proof}
Let $(J,I)$ be a deforming-Nijenhuis pair for $\Theta$. Combining Corollary \ref{Nijenhuis}, Theorem \ref{new thma} and Proposition \ref{Jdeforming_n}, we have that $(J,I)$ is a deforming-Nijenhuis pair for $\Theta_k$, for all $k \in \Nn$. From Proposition \ref{Nijenhuis_theta_m}, we get that $I^{2n+1}$ is Nijenhuis for $\Theta_k$, for all $k,n \in \Nn$. Since $I$ and $J$ anti-commute, the tensors $I^{2n+1}$ and $J$ also anti-commute and, from Theorem \ref{propconcgeral}, we have that $C_{\Theta_{k}}(I^{2n+1}, J)=0$, for all $k,n \in \Nn$. Thus, $(J,I^{2n+1})$ is a deforming-Nijenhuis pair for $\Theta_k$, for all $k,n \in \Nn$.
\end{proof}



\subsection{Hierarchy of Poisson-Nijenhuis pairs for $\Theta_k$, $k \in \Nn$}

We introduce the notions of Poisson tensor, Poisson-Nijenhuis pair and compatible Poisson tensors for a pre-Courant algebroid $(E,\Theta)$ and construct a hierarchy of Poisson-Nijenhuis pairs.

\

We start by introducing the  notion of Poisson tensor.

\begin{defn}
A skew-symmetric $(1,1)$-tensor $J$ on a pre-Courant algebroid $(E,\Theta)$ satisfying $\Theta_{J,J}=0$ is said to be a {\em Poisson} tensor for $\Theta$.
\end{defn}

In the next example, we show that the previous definition extends the usual definition of a Poisson bivector on a Lie algebroid.

\begin{ex}  \label{examplePoisson}
Let $(A, \mu)$ be a Lie algebroid. Consider the Courant algebroid $(A\oplus A^*, \Theta=\mu)$ and the $(1,1)$-tensor  $J_\pi$ of Example \ref{example1} a). Then,  $J_\pi$ is a Poisson tensor for $\Theta=\mu$ if and only if $\pi$ is a Poisson tensor on $(A, \mu)$:
$$\Theta_{J_\pi,J_\pi}=0 \Leftrightarrow \{\pi, \{\pi, \mu\}\}=0 \Leftrightarrow [\pi,\pi]_\mu=0.$$
\vspace*{-1cm}
\flushright{$\diamondsuit$}
\end{ex}

\begin{ex}
The operators introduced in example \ref{ex:central} are Poisson operators on Lie algebras.
\end{ex}

The next theorem follows directly from Lemma~\ref{theta_n_JJI} c).
\begin{thm}\label{J_Poisson_for_Theta_n}
Let $I$ and $J$ be two skew-symmetric $(1,1)$-tensors on a pre-Courant algebroid $(E,\Theta)$ such that $(I,J)$ is a compatible pair w.r.t. $\Theta$,  $\Theta_{\{J,\{I,J\}\}}=0$ and ${\mathcal T}_\Theta I(JX,Y)={\mathcal T}_\Theta I (X,JY)=0$, for all sections $X$ and $Y$ of $E$. If $J$ is Poisson for $\Theta$, then $J$ is Poisson for $\Theta_k$, for all $k \in \Nn$.
\end{thm}

\

Requiring  $\Theta_{\{J,\{I,J\}\}}=0$ might seem a bit arbitrary, but it is not. In fact, in the case where $I$ and $J$ anti-commute, this condition may be interpreted as $ I \smc J^2 $ being a $\Theta$-cocycle. When $E=A \oplus A^*$, a $(1,1)$-tensor $J_\pi$ of the type considered in Example~\ref{example1} a) satisfies trivially this condition because $J_{\pi}^2=0$.

\

Now, we introduce the main notion of this section.

\begin{defn} \label{PNpair}
Let $I$ and $J$ be two skew-symmetric $(1,1)$-tensors on a pre-Courant algebroid $(E,\Theta)$. The pair $(J,I)$ is said to be a {\em Poisson-Nijenhuis pair} for $\Theta$ if
\begin{itemize}
\item $(J,I)$ is a compatible pair w.r.t. $\Theta$;
\item $J$ is Poisson for $\Theta$;
\item $I$ is Nijenhuis for $\Theta$.
\end{itemize}
\end{defn}

\begin{rem}
If $(J,I)$ is a Poisson-Nijenhuis pair for $\Theta$, then it is a deforming-Nijenhuis pair for $\Theta$.
\end{rem}

Recall that a Poisson-Nijenhuis structure on a Lie algebroid $(A, \mu)$ is a pair $(\pi, N)$, where $\pi$ is a Poisson bivector and $N:A \to A$ is a Nijenhuis tensor such that $N \pi^\#= \pi^\# N^*$ and $C_{\mu}(\pi,N)=0$.

The next example shows the relation between Definition \ref{PNpair} and the notion of Poisson-Nijenhuis structure on a Lie algebroid.
\begin{ex}  \label{examplePN}
Let $(\pi,N)$ be a Poisson-Nijenhuis structure on a Lie algebroid $(A, \mu)$, with $N^2= \alpha id_A$, $\lambda \in \Rr$. Consider the Courant algebroid $(E, \Theta)$, with $E=A\oplus A^*$ and $\Theta=\mu$,  $J_\pi$ and $I_N$  as in  Example \ref{example1} a) and d), respectively. Then, $(J_\pi,I_N)$ is a Poisson-Nijenhuis pair for $\Theta$.
In fact,  $N \pi^\#= \pi^\# N^* \Leftrightarrow I_N \smc J_\pi=- J_\pi \smc I_N$ and $C_{\mu}(\pi,N) = C_{\mu}(J_\pi,I_N)=0$, so that $(J_\pi,I_N)$ is a compatible pair w.r.t. $\mu$. Moreover, $\pi$ is a Poisson bivector on $(A, \mu)$ if and only if $J_\pi$ is Poisson for $\Theta=\mu$ (see Example \ref{examplePoisson}) and $I_N$ is Nijenhuis  for $\Theta=\mu$ (see Example \ref{example1} d)). The above arguments show that conversely, if $(J_\pi,I_N)$ is a Poisson-Nijenhuis pair for $\Theta=\mu$ with   $N^2= \alpha id_A$, then $(\pi,N)$ is a Poisson-Nijenhuis structure on  $(A, \mu)$.

\flushright{$\diamondsuit$}
\end{ex}

Now, we introduce the notion of compatible Poisson tensors.

\begin{defn}
Let $J$ and $J'$ be two Poisson tensors for the pre-Courant structure $\Theta$ on the vector bundle $(E, \langle.,.\rangle)$. The tensors $J$ and $J'$ are said to be {\em compatible} Poisson tensors for $\Theta$ if $J+J'$ is a Poisson tensor for $\Theta$, i.e, $\Theta_{J+J',J+J'}=0$.
\end{defn}

An immediate consequence of this definition is the following:

\begin{lem}
Let $J$ and $J'$ be two Poisson tensors for $\Theta$. Then,  $J$ and $J'$ are compatible Poisson tensors for $\Theta$ if and only if $\Theta_{J,J'}+\Theta_{J',J}=0$. In other words, $J$ and $J'$ are compatible Poisson tensors for $\Theta$ if and only if $J$ and $J'$ anti-commute w.r.t. $\Theta$.
\end{lem}

\begin{ex}
Let $(A, \mu)$ be a Lie algebroid, consider the Courant algebroid $(A\oplus A^*, \Theta=\mu)$ and take two Poisson tensors for $\Theta=\mu$, $J_\pi$ and $J_{\pi'}$, of the type considered in Example~\ref{example1} a). Then,
\begin{align*}
\Theta_{J_{\pi},J_{\pi'}}+\Theta_{J_{\pi'},J_{\pi}}=0  & \Leftrightarrow \{\pi', \{\pi, \mu\}\} + \{\pi, \{\pi', \mu\}\}=0\\
 & \Leftrightarrow 2 \{\pi', \{\pi, \mu\}\}=0\\ & \Leftrightarrow [\pi,\pi']_\mu=0,
\end{align*}
and we recover the notion of compatible Poisson tensors on a Lie algebroid.
\flushright{$\diamondsuit$}
\end{ex}

In order to construct a hierarchy of Poisson-Nijenhuis pairs, we need the next proposition.

\begin{prop}\label{prop_T(ThetaJ)I}
Let $I$ and $J$ be two anti-commuting skew-symmetric $(1,1)$-tensors on a pre-Courant algebroid $(E, \Theta)$. Then, for all sections $X$ and $Y$ of $E$,
\begin{equation}  \label{torsaoJI2}
{\mathcal T}_{\Theta_I}J(X,Y)=-J(C_\Theta (I,J)(X,Y))- {\mathcal T}_{\Theta}J(IX,Y)- {\mathcal T}_{\Theta}J(X,IY)-I({\mathcal T}_{\Theta}J(X,Y))
\end{equation}
and
\begin{equation}  \label{torsaoJI1}
{\mathcal T}_{\Theta_J}I(X,Y)=-I(C_\Theta (I,J)(X,Y))- {\mathcal T}_{\Theta}I(JX,Y)- {\mathcal T}_{\Theta}I(X,JY)-J({\mathcal T}_{\Theta}I(X,Y)).
\end{equation}
\end{prop}

\begin{proof}
Since the roles of $I$ and $J$ can be reversed, we only prove (\ref{torsaoJI2}). We compute ${\mathcal T}_{\Theta_I}J$ and $C_\Theta (I,J)$. For any sections $X,Y$ of $E$, we have
\[
\begin{array}{ll}
& {\mathcal T}_{\Theta_I}J(X,Y)= [JX,JY]_I-J[JX,Y]_I-J[X,JY]_I+J^2[X,Y]_I \\
&= [IJX, JY]+[JX,IJY]-I[JX,JY]
  - J[IJX,Y]-J[JX,IY]+JI[JX,Y] \\
& -J[IX,JY]-J[X,IJY]+JI[X,JY]
 + J^2 [IX,Y]+J^2 [X,IY]-J^2I[X,Y]
\end{array}
\]
and
\begin{equation*}
C_\Theta (I,J)(X,Y)= 2([JX,IY]+[IX,JY]-I([JX,Y]+[X,JY])-J([IX,Y]+[X,IY])).
\end{equation*}

Thus,
\[
\begin{array} {ll}
&{\mathcal T}_{\Theta_I}J(X,Y) +J(C_\Theta (I,J)(X,Y)) = -[JIX,JY]  -[JX,JIY]  -I[JX,JY] \\
& +J[JIX,Y]  +J[JX,IY]  +IJ[JX,Y] +J[IX,JY]  +J[X,JIY]  +IJ[X,JY] \\
& - J^2[IX,Y] -J^2[X,IY]  -IJ^2[X,Y] \\
& = - {\mathcal T}_{\Theta}J(IX,Y)  - {\mathcal T}_{\Theta}J(X,IY)  -I({\mathcal T}_{\Theta}J(X,Y)).
\end{array}
\]
\end{proof}

The next theorem defines a hierarchy of Poisson-Nijenhuis pairs.

\begin{thm}  \label{generalPNhierarchy}
Let $I$ and $J$ be two skew-symmetric $(1,1)$-tensors on a pre-Courant algebroid $(E,\Theta)$, such that $(J,I)$ is a Poisson-Nijenhuis pair for $\Theta$ and $\Theta_{\{J,\{I,J\}\}}=0$. Then,
 \begin{enumerate}
\item
  $I^n \smc J$ is a Poisson tensor for $\Theta_k$, for all $n, k \in \Nn$;
  \item
  $(I^n \smc J)_{n \in \Nn}$ is a hierarchy of pairwise compatible Poisson tensors for $\Theta_k$, for all $k \in \Nn$;
  \item
 $(I^n \smc J, I^{2m+1})$ is a Poisson-Nijenhuis pair for $\Theta_k$, for all $m,n,k \in \Nn$.
 \end{enumerate}
\end{thm}

\

The proof of the above theorem needs two auxiliary lemmas.

\begin{lem}
Let $I$ and $J$ be two skew-symmetric $(1,1)$-tensors on a pre-Courant algebroid $(E,\Theta)$, such that $(I,J)$ is a compatible pair w.r.t. $\Theta$. If $I$ is Nijenhuis for $\Theta$, then $I$ is Nijenhuis for $(\Theta_k)_J$, for all $k \in \mathbb{N}$.
\end{lem}
\begin{proof}
Fix $k \in \mathbb{N}$. From Corollary~\ref{Nijenhuis}, $I$ is Nijenhuis for $\Theta_k$. Also, applying Theorem~\ref{new thma}, we get $C_{\Theta_k}(I,J)=0$. Finally, using (\ref{torsaoJI1}) for the pre-Courant structure $\Theta_k$, we conclude that $I$ is Nijenhuis for $(\Theta_k)_J$.
\end{proof}

\begin{lem}
Let $I$ and $J$ be two skew-symmetric $(1,1)$-tensors on a pre-Courant algebroid $(E,\Theta)$ such that $J$ is Poisson for $\Theta$, $\Theta_{\{J,\{I,J\}\}}=0$ and ${\mathcal T}_\Theta I(JX,Y)={\mathcal T}_\Theta I (X,JY)=0$, for all sections $X$ and $Y$ of $E$. If $(I,J)$ is a compatible pair w.r.t. $\Theta$, then $(I,J)$ is a compatible pair w.r.t. $(\Theta_k)_J$, for all $k \in \mathbb{N}$.
\end{lem}
\begin{proof}
Fix $k \in \mathbb{N}$. By definition, $C_{(\Theta_k)_J}(I,J)=(\Theta_k)_{J,I,J} + (\Theta_k)_{J,J,I}$. In order to compute $(\Theta_k)_{J,I,J}$, remember formula~(\ref{thetaJIJ}) for the pre-Courant structure $\Theta_k$:
$$(\Theta_k)_{J,I,J}= \frac{1}{3} \left( (\Theta_k)_{J,J,I}+ (\Theta_k)_{\{J,\{I,J\}\}}+ \{J, C_{\Theta_k} (I,J)\} \right).$$
Since $(I,J)$ is a compatible pair w.r.t. $\Theta$, applying Theorem~\ref{new thma}, we get $C_{\Theta_k}(I,J)=0$. Furthermore, from Lemma~\ref{theta_cocycle}(ii), we have $\left(\Theta_k\right)_{\{J,\{I,J\}\}}=0$. Then, the formula above turns into
$(\Theta_k)_{J,I,J}= \frac{1}{3} (\Theta_k)_{J,J,I},$
so that
$C_{(\Theta_k)_J}(I,J)=\frac{4}{3} (\Theta_k)_{J,J,I}.$

Now, using Theorem~\ref{J_Poisson_for_Theta_n}, we get $(\Theta_k)_{J,J,I}=0$. Therefore, $(I,J)$ is a compatible pair w.r.t. $(\Theta_k)_J$.
\end{proof}

We address now the proof of the above theorem.
\begin{proof}[Proof of Theorem~\ref{generalPNhierarchy}]
Let $(I,J)$ be a Poisson-Nijenhuis pair for $\Theta$ such that $\Theta_{\{J,\{I,J\}\}}=0$. We start by proving that
\begin{equation} \label{compatible Poisson}
(\Theta_k)_{I^m\smc J, I^n\smc J}=0,
\end{equation}
for all $m,n,k \in \Nn$.
From the above auxiliary lemmas, $(I,J)$ is a compatible pair w.r.t. $(\Theta_{k+m})_J$ and $I$ is Nijenhuis for $(\Theta_{k+m})_J$. Then, using Lemma~\ref{lem20} for the pre-Courant structure $(\Theta_{k+m})_J$, we obtain
\begin{align*}
(\Theta_k)_{I^m\smc J, I^n\smc J} &=\left((\Theta_{k+m})_J\right)_{I^n\smc J}
    =\left((\Theta_{k+m})_J\right)_{{\scriptsize \underbrace{I, \dots,I}_{n}},J} \nonumber \\
   & =\Theta_{{\scriptsize \underbrace{I, \dots,I}_{k+m}},J,{\scriptsize \underbrace{I, \dots,I}_{n}},J}=(-1)^n\,\Theta_{{\scriptsize \underbrace{I, \dots,I}_{k+m+n}},J,J},
\end{align*}
where in the last equality  we used $n$ times $C_{\Theta_{s}}(I,J)=0$, for all $s \in \Nn$ (see Theorem~\ref{new thma}).
Using Theorem~\ref{J_Poisson_for_Theta_n}, we obtain (\ref{compatible Poisson}), from where statements (1) and (2) follow.
From Theorem~\ref{propconcgeral}, $(I^n \smc J, I^{2m+1})$ is a compatible pair w.r.t. $\Theta_k$ and, from Proposition~\ref{Nijenhuis_theta_m}, $I^{2m+1}$ is Nijenhuis for $\Theta_k$. Combining this with statement (1), we get statement (3).
\end{proof}

Using the Poisson-Nijenhuis pair arising from a Poisson-Nijenhuis structure as in Example~\ref{examplePN}, we recover most of the hierarchy already studied by \cite{magriYKS}, up to a minor difference. In this general setting it is not possible to consider $I^{2n}$ since it is not a skew-symmetric $(1,1)$-tensor.

\

To conclude this section, we come back to the deforming-Nijenhuis pairs to discuss a particular case.

\begin{prop}
Let $I$ and $J$ be two skew-symmetric $(1,1)$-tensors on a pre-Courant algebroid $(E,\Theta)$, such that  $I^2=\alpha \, id_E$ and $\Theta_{\{J, \{I,J\}\}}=\lambda_0 \Theta_{J,J,I}$, for some $\alpha, \lambda_0 \in \Rr$. If $(J,I)$ is a deforming-Nijenhuis pair for $\Theta$, then $(I^n \smc J,I)$ is a deforming-Nijenhuis pair for $\Theta$, for all $n \in \Nn$.
\end{prop}
\begin{proof}
Let $(J,I)$ be a deforming-Nijenhuis pair for $\Theta$. First, we prove that $I^n \smc J$ is deforming for $\Theta$.
Since $I^2=\alpha \, id_E$, $I^n \smc J$ is proportional either to $J$ or to $I \smc J$. So, we only need to prove that $I \smc J$ is deforming for $\Theta$. Using Lemma \ref{IoJ} and because $I$ and $J$ anti-commute, we have
\begin{equation*}
    \Theta_{I \smc J, I \smc J}= \Theta_{I, J, I \smc J}=\frac{1}{2} \Theta_{I, J, \{J,I\}}
    =\frac{1}{2} \left(\Theta_{I,J,I,J}-\Theta_{I,J,J,I}\right),
\end{equation*}
where in the last equality we used the Jacobi identity. Using (\ref{thetaJIJ}) for $\Theta_I$ and Lemma \ref{theta_cocycle}, we get
\begin{equation*}
    2 \Theta_{I \smc J, I \smc J}=\frac{1}{3} \left(\Theta_{I,J,J,I} + \Theta_{I,\{J, \{I,J\}\}}\right)-\Theta_{I,J,J,I}
    =-\frac{2}{3} \Theta_{I,J,J,I} + \frac{1}{3} \Theta_{\{J, \{I,J\}\},I}.
\end{equation*}
Now, from the equality (\ref{thetaIJJ}), we obtain
$$2 \Theta_{I \smc J, I \smc J}=\frac{2}{9} \Theta_{J,J,I,I} + \frac{5}{9} \Theta_{\{J, \{I,J\}\},I}.$$
Since $\Theta_{\{J, \{I,J\}\}}=\lambda_0 \Theta_{J,J,I}$ and $\Theta_{J,J}=\eta\,  \Theta$, for some $\eta \in \Rr$, we get
$$\Theta_{I \smc J, I \smc J}=\frac{2+5\lambda_0}{18}\eta\, \Theta_{I,I}=\frac{2+5\lambda_0}{18}\eta \, \Theta_{I^2}=\frac{2+5\lambda_0}{18}\eta \alpha\, \Theta,$$
where, in the last equalities, we used the fact that $I$ is Nijenhuis and satisfies $I^2=\alpha \, id_E$. Therefore, $I \smc J$ is deforming for $\Theta$.

The tensors $I$ and $I^n \smc J$ anti-commute and, from Theorem~\ref{propconcIn}, $C_\Theta(I,I^n \smc J)=0$. Thus, $(I^n \smc J,I)$ is a deforming-Nijenhuis pair for $\Theta$.
\end{proof}

\

Notice that $(I^n \smc J,I)_{n \in \Nn}$ is a very poor hierarchy of deforming-Nijenhuis pairs since, as we already mentioned,  all the pairs are of type either  $(J,I)$ or  $(I \smc J,I)$. In fact
we have, for all $n \in \Nn$,
 $$I^{2n} \smc J= \alpha^n J, \quad I^{2n+1} \smc J= \alpha^n I \smc J.$$


\section{Hierarchies of Nijenhuis pairs}

The last part of this article is devoted to the study of pairs of Nijenhuis tensors on pre-Courant algebroids.

\subsection{Nijenhuis pair for a hierarchy of pre-Courant structures}
We introduce the notion of Nijenhuis pair for a pre-Courant algebroid $\Theta$ and prove that a Nijenhuis pair $(I,J)$ for  $\Theta$ is still a Nijenhuis pair for any deformation of $\Theta$, either by $I$ or $J$.

We start by introducing the notion of Nijenhuis pair for a pre-Courant algebroid.

\begin{defn}
Let $I$ and $J$ be two skew-symmetric tensors on  a pre-Courant algebroid $(E, \Theta)$.
The pair $(I,J)$ is called a {\em Nijenhuis pair } for $\Theta$, if it is a compatible pair w.r.t. $\Theta$ and $I$ and $J$ are both Nijenhuis for $\Theta$.
\end{defn}

\begin{ex}
Let $J$ be a deforming tensor on $(E,\Theta)$, i.e. $\Theta_{J,J}=\eta \, \Theta$, for some $\eta \in \Rr$. If $(J,I)$ is a deforming-Nijenhuis pair, with $J^2 = \eta \, id_E$,  then $(J,I)$ is a Nijenhuis pair. In particular, if $(J,I)$ is Poisson-Nijenhuis pair, and $J^2=0$, then $(J,I)$ is a Nijenhuis pair. Notice that this happens when $J=J_\pi$ as in Example~\ref{example1}.
\flushright{$\diamondsuit$}
\end{ex}

In the next proposition we compute the torsion of the composition $I \smc J$.

\begin{prop}  \label{torsionIoJ}
    Let $I$ and $J$ be two anti-commuting tensors on a pre-Courant algebroid $(E, \Theta)$. Then, for all sections $X$ and $Y$ of $E$,
    \begin{multline}\label{torsion_of_composition}
        2 {\mathcal T}_{\Theta}(I \smc J)(X,Y) = \bigg({\mathcal T}_{\Theta}I(JX, JY) - J\left({\mathcal T}_{\Theta}I(JX,Y) + {\mathcal T}_{\Theta}I(X,JY)\right) -\\- J^2({\mathcal T}_{\Theta}I(X,Y)) \bigg) + \underset{I,J}{\circlearrowleft},
    \end{multline}
 where $ \underset{I,J}{\circlearrowleft}$ stands for permutation of $I$ and $J$.
\end{prop}
\begin{proof}
Let us compute the first four terms of the right hand side of equation (\ref{torsion_of_composition}):
\begin{align*}
    &\phantom{--}{\mathcal T}_{\Theta}I(JX,JY)&=& \phantom{--}[IJX,IJY] &-& I[IJX,JY] &-& I[JX,IJY] &+& I^2[JX,JY]\\
    &-J\left({\mathcal T}_{\Theta}I(JX,Y)\right)&=& -J[IJX,IY] &+& JI[IJX,Y] &+& JI[JX,IY] &-& JI^2[JX,Y]\\
    &-J\left({\mathcal T}_{\Theta}I(X,JY)\right)&=& -J[IX,IJY] &+& JI[IX,JY] &+& JI[X,IJY] &-& JI^2[X,JY]\\
    &-J^2\left({\mathcal T}_{\Theta}I(X,Y)\right)&=& -J^2[IX,IY] &+& J^2I[IX,Y] &+& J^2I[X,IY] &-& J^2I^2[X,Y].\\
\end{align*}
The terms appearing on the right hand sides of the above equalities can be addressed in a matrix form:
$$M(I,J)(X,Y)=
\left[\begin{array}{cccc}
    [IJX,IJY] &- I[IJX,JY] &- I[JX,IJY] & I^2[JX,JY]\\
    -J[IJX,IY] & JI[IJX,Y] & JI[JX,IY] &- JI^2[JX,Y]\\
    -J[IX,IJY] & JI[IX,JY] & JI[X,IJY] &- JI^2[X,JY]\\
    -J^2[IX,IY] & J^2I[IX,Y] & J^2I[X,IY] &- J^2I^2[X,Y]\\
\end{array}\right].
$$
Because $I$ and $J$ anti-commute, intertwining the tensors $I$ and $J$, we obtain the matrix $M(J,I)$ with entries given by
$$M(J,I)_{m,n}=\left\{
  \begin{array}{ll}
    -M(I,J)_{n,m},&\textrm{ if } m\neq n\\
    M(I,J)_{m,n},&\textrm{ if } m=n
  \end{array}
\right.$$
for all $m,n=1,\ldots, 4$.

Note that the right hand side of equation (\ref{torsion_of_composition}) is the sum of all the entries of both matrices $M(I,J)(X,Y)$ and $M(J,I)(X,Y)$. Thus,
\begin{align*}
    {\mathcal T}_{\Theta}I(JX, JY) &- J\left({\mathcal T}_{\Theta}I(JX,Y) + {\mathcal T}_{\Theta}I(X,JY)\right) - J^2({\mathcal T}_{\Theta}I(X,Y)) + \underset{I,J}{\circlearrowleft}=\\
    &=2\left([IJX,IJY] + JI[IJX,Y] + JI[X,IJY] - J^2I^2[X,Y]\right)\\
    &=2\left([IJX,IJY] - IJ[IJX,Y] - IJ[X,IJY] + (IJ)^2[X,Y]\right)\\
    &=2 {\mathcal T}_{\Theta}(I \smc J)(X,Y),
\end{align*}
and the proof is complete.
\end{proof}

\begin{prop}\label{Prop:arealsonijenhuis}
Let $I$ and $J$ be two skew-symmetric tensors on  a pre-Courant algebroid $(E, \Theta)$. If $(I,J)$ is a Nijenhuis pair for $\Theta$, then $(I, I \smc J)$ and $(J,I \smc J)$ are also Nijenhuis pairs for $\Theta$.
\end{prop}

\begin{proof}
It is obvious that $I$ and $I \smc J$ anti-commute, as well as $J$ and $I \smc J$. From (\ref{torsion_of_composition})  we conclude that $I \smc J$ is a Nijenhuis tensor and from (\ref{concomitant_recursion}), with $n=1$, we get $C_\Theta(I, I \smc J)=C_\Theta(J, I \smc J)=0$.
\end{proof}

\

Using Proposition \ref{Prop:arealsonijenhuis}, we may establish a relationship between Nijenhuis pairs and hypercomplex triples.

\

The triple $(I,J,K)$ of skew-symmetric $(1,1)$-tensors on a pre-Courant algebroid $(E, \Theta)$ is called a {\em hypercomplex triple} if $I^2=J^2=K^2=I\smc J\smc K=-id_E$ and all the six Nijenhuis concomitants
${\mathcal N}_\Theta(I,I)$, ${\mathcal N}_\Theta(J,J)$, ${\mathcal N}_\Theta(K,K)$, ${\mathcal N}_\Theta(I,J)$, ${\mathcal N}_\Theta(J,K)$ and ${\mathcal N}_\Theta(I,K)$ vanish \cite{Stienon}. (See (\ref{torsion_sum}) for the definition of ${\mathcal N}_\Theta$).

\begin{ex}
Given a Nijenhuis pair $(I,J)$ such that $I^2=J^2=-id_E$, the triple $(I, J, I\smc J)$ is a hypercomplex structure. Conversely, for every hypercomplex structure $(I,J,K)$, the pairs $(I,J)$, $(J,K)$ and $(K,I)$ are Nijenhuis pairs.
\flushright{$\diamondsuit$}
\end{ex}

 The main result of this section is the following.

\begin{thm} \label{NijPair}
Let $I$ and $J$ be two $(1,1)$-tensors on a pre-Courant algebroid $(E, \Theta)$. If $(I,J)$ is a Nijenhuis pair for $\Theta$, then $(I,J)$ is a Nijenhuis pair for $\Theta_{T_1, T_2, \dots, T_s}$, for all $ s \in \Nn$, where $T _i$ stands either for $I$ or for $J$, for every  $i=1, \dots, s$.
\end{thm}

\begin{proof}
Let $(I,J)$ be a Nijenhuis pair for $\Theta$. Combining formulae (\ref{concomitant_recursion}) and (\ref{C_theta2}), we get
\begin{equation}  \label{concthetaI}
C_{\Theta_I}(I,J)(X,Y)=  2I(C_{\Theta}(I,J)(X,Y))+ 4\,{\mathcal T}_{\Theta}I(JX,Y)+4\,{\mathcal T}_{\Theta}I(X,JY)=0.
\end{equation}

Now, from Corollary~\ref{Nijenhuis}, (\ref{torsaoJI2})  and (\ref{concthetaI}), we conclude that $(I,J)$ is a Nijenhuis pair for $\Theta_I$. Since we may exchange the roles of $I$ and $J$, we also conclude that $(I,J)$ is a Nijenhuis pair for $\Theta_J$.

Since the Corollary \ref{Nijenhuis} and the formulae (\ref{torsaoJI2}) and (\ref{concthetaI}) hold for any anti-commuting tensors $I$ and $J$,  and for any pre-Courant structure $\Theta$ on $E$, we can repeat the previous argument iteratively to conclude that $(I,J)$ is a Nijenhuis pair for $\Theta_{T_1, T_2, \dots, T_s}$, for all $ s \in \Nn$, where $T_i$ stands either for $I$ or for $J$, for every  $i=1, \dots, s$.
\end{proof}

As a consequence of the above theorem, we deduce:

\begin{cor}
Let $I$ and $J$ be two $(1,1)$-tensors on a Courant algebroid $(E, \Theta)$. If $(I,J)$ is a Nijenhuis pair for $\Theta$ then, for all $ s \in \Nn$,  $\Theta_{T_1, T_2, \dots, T_s}$ is a Courant structure on $E$, where $T _i$ stands either for $I$ or for $J$, for every  $i=1, \dots, s$.
\end{cor}

\subsection{Hierarchies of Nijenhuis pairs}
Starting with a Nijenhuis pair $(I,J)$ for a pre-Courant algebroid $(E, \Theta)$, we construct several hierarchies of Nijenhuis pairs for any deformation of $\Theta$, either by $I$ or $J$.

We start with the construction of a hierarchy $(I^{2m+1},J)_{m \in \Nn}$ of Nijenhuis pairs where one of the Nijenhuis tensors keeps unchanged.

\begin{prop}  \label{oddpowers}
Let $I$ and $J$ be two $(1,1)$-tensors on a pre-Courant algebroid $(E,\Theta)$. If $(I,J)$ is a Nijenhuis pair for $\Theta$ then, for all $m \in \Nn$, $(I^{2m+1},J)$ is a Nijenhuis pair for $\Theta_{T_1, T_2, \dots, T_s}$, for all $ s \in \Nn$, where $T_i$ stands either for $I$ or for $J$, for every  $i=1, \dots, s$.
\end{prop}

\begin{proof}
The proof follows from Corollary~\ref{hierarchy_Nijenhuis},
 Theorem~\ref{propconcgeral} and Theorem~\ref{NijPair}.
\end{proof}

Now we consider the hierarchy $(I^{2m+1},J^{2n+1})_{m,n \in \mathbb{N}}$.
This case follows from the previous one: for every $m \in \Nn$,  $(I^{2m+1},J)$ is a Nijenhuis pair. Applying Proposition~\ref{oddpowers} to each one of these pairs, we get that $(I^{2m+1},J^{2n+1})_{m,n \in \mathbb{N}}$ is a hierarchy of  Nijenhuis pairs and we end up with the following.

\begin{thm}
Let $I$ and $J$ be two $(1,1)$-tensors on a pre-Courant algebroid $(E,\Theta)$. If $(I,J)$ is a Nijenhuis pair for $\Theta$ then, for all $m, n \in \Nn$, $(I^{2m+1},J^{2n+1})$ is a Nijenhuis pair for  $\Theta_{T_1, T_2, \dots, T_s}$, for all $ s \in \Nn$, where $T_i$ stands either for $I$ or for $J$, for every  $i=1, \dots, s$.
\end{thm}

\

Let $I$ and $J$ be two skew-symmetric $(1,1)$-tensors on a pre-Courant algebroid $(E,\Theta)$. If $I$ and $J$ are Nijenhuis tensors, we know (see Corollary \ref{hierarchy_Nijenhuis}) that, for any $m,n \in \mathbb{N}$, $I^m$ and $J^n$ are also Nijenhuis tensors for $\Theta$. The next lemma gives a condition granting that $I^m \smc J^n$ is also Nijenhuis.

\begin{lem}  \label{Nijenhuis_odd_power}
 Let $I$ and $J$ be two skew-symmetric $(1,1)$-tensors on a pre-Courant algebroid $(E,\Theta)$.   If $I$ and $J$ are anti-commuting Nijenhuis tensors, then $I^m \smc J^n$ is a Nijenhuis tensor provided that one at least of the integers $m,n$ is odd.
\end{lem}
\begin{proof}
    As the roles of the tensors $I$ and $J$ are symmetric, we can suppose that $m$ is odd (and $n$ is even or odd).
If $n$ is also odd then $I^m$ and $J^n$ anti-commute and the result follows from Proposition~\ref{torsionIoJ}.
Suppose now that $m$ is odd and $n$ is even. By the previous case, $I^m\smc J^{n-1}$ is Nijenhuis and anti-commutes with $J$:
$$(I^m\smc J^{n-1})\smc J=I^m\smc J^n=-J \smc (I^m\smc J^{n-1}).$$
Then, using again Proposition~\ref{torsionIoJ}, we conclude that $I^m \smc J^n$ is a Nijenhuis tensor.
\end{proof}

The main result of this section is the following theorem.

\begin{thm} \label{gen_hierarchy}
Let $I$ and $J$ be two skew-symmetric $(1,1)$-tensors on a pre-Courant algebroid $(E, \Theta)$. If $(I,J)$ is a Nijenhuis pair for $\Theta$, then for all $m,n,t \in \Nn$, $(I^{2m+1}\smc J^n, J^{2t+1})$ is a Nijenhuis pair for $\Theta_{T_1, T_2, \dots, T_s}$, for all $ s \in \Nn$, where $T_i$ stands either for $I$ or for $J$, for every  $i=1, \dots, s$.
\end{thm}

\begin{proof} First, we prove that $(I^{2m+1}\smc J^n, J^{2t+1})$ is a Nijenhuis pair for $\Theta$, for all $m,n,t \in \mathbb{N}$.
We already know that $I^{2m+1}\smc J^n$ is Nijenhuis (see Lemma~\ref{Nijenhuis_odd_power}) and that $J^{2t+1}$ is Nijenhuis (see Corollary \ref{hierarchy_Nijenhuis}). Moreover, $I^{2m+1}\smc J^n$ anti-commutes with  $J^{2t+1}$ and, applying (\ref{item_c}), we get $C_{\Theta}(I^{2m+1}\smc J^n, J^{2t+1})=0$.

Using Theorem~\ref{NijPair}, this result can be extended to all pre-Courant structures $\Theta_{T_1, T_2, \dots, T_s}$, where $T_i$ stands either for $I$ or for $J$, for every  $i=1, \dots, s$.
\end{proof}

\

\noindent {\bf Acknowledgments.} This work was partially supported by CMUC-FCT (Portugal) and FCT grant PTDC/MAT/099880/2008 through
European program COMPETE/FEDER.

\end{document}